\documentclass[onefignum,onetabnum]{siamart171218}

\newsiamremark{remark}{Remark}
\newsiamremark{hypothesis}{Hypothesis}
\crefname{hypothesis}{Hypothesis}{Hypotheses}
\newsiamthm{claim}{Claim}

\usepackage[utf8]{inputenc}
\usepackage[english]{babel}
\usepackage{amsmath}
\usepackage{amsfonts, mathrsfs}
\usepackage{amssymb}
\usepackage{dsfont}
\usepackage{hyperref,polynom}
\usepackage{xfrac}
\usepackage{fonts,esint}
\usepackage{cite}
\usepackage{wrapfig} 
\usepackage[export]{adjustbox} 
\usepackage{todonotes}
\usepackage{mathtools} 
\usepackage{graphicx} 
\usepackage{array} 
\usepackage{enumerate} 

\makeatletter
\def\namedlabel#1#2{\begingroup
    #2%
    \def\@currentlabel{#2}%
    \phantomsection\label{#1}\endgroup
}
\makeatother

\newcommand{\e}{\mathrm{e}}

\def\quand{\quad \mbox{and} \quad}

\newcommand{\intdm}[3]{ \int_{#1} #2 \, \mathrm{d}#3}

\newcommand{\intdmt}[4]{\displaystyle \int_{#1}^{#2} #3 \, \mathrm{d}#4}
\newcommand{\iintdmt}[6]{\displaystyle \int_{#1}^{#2} \int_{#3} #4 \, \mathrm{d}#5 \, \mathrm{d}#6}

\newcommand{\Vint}[1]{\left\langle #1 \right\rangle}

\newcommand{\Vnorm}[1]{\left\Vert #1 \right\Vert}
\def\p{\partial}
\newcommand{\EquationReference}[2]{\mathrel{\overset{\makebox[0pt]{\mbox{\normalfont\tiny\sffamily #1}}}{#2}}}
\def\veps{\varepsilon}
\def\pv{\text{P.V.}}

\def\grad{\nabla}

\newcommand{\bfs}[1]{\boldsymbol{#1}}
\newcommand{\bdx}{\mathbf{x}}

\newcommand{\bdy}{{\bf y}}
\newcommand{\bfxi}{\boldsymbol{\xi}}

\usepackage{eurosym}
\everymath{\displaystyle}

\headers{Higher integrability of Solutions to linear Peridynamic models}{J. Scott, T. Mengesha}
\title{A Potential Space Estimate for Solutions of Systems of Nonlocal Equations in Peridynamics\thanks{Submitted to the editors DATE.
\funding{This research is supported by the U.S. NSF grant DMS-1615726. }}}

\author{James Scott \thanks{University of Tennessee Knoxville, Knoxville, TN 
  (\email{jscott66@vols.utk.edu})}
\and Tadele Mengesha \thanks{University of Tennessee Knoxville, Knoxville, TN  
  (\email{mengesha@utk.edu})}}

\begin{document}

\maketitle

\begin{abstract}
 We show that weak solutions to the strongly-coupled system of nonlocal equations of linearized peridynamics belong to a potential space with higher integrability. Specifically, we show that a function measuring local fractional derivatives of weak solutions to a linear system belongs to $L^p$ for some $p > 2$ with no additional assumptions other than measurability and ellipticity of the coefficients. This is a nonlocal analogue of an inequality of Meyers for weak solutions to an elliptic system of equations. We also show that functions in $L^p$ whose Marcinkiewicz-type integrals are in $L^p$ in fact belong to the Bessel potential space $\mathcal{L}^{p}_s$. Thus the fractional analogue of higher integrability of the solution's gradient is  displayed explicitly. The distinction here is that the Marcinkiewicz-type integral exhibits the coupling from the nonlocal model and does not resemble other classes of potential-type integrals found in the literature.
\end{abstract}



\section{Introduction}

Nonlocal models are becoming commonplace across application areas. Typically these models involve averaged difference quotients instead of derivatives of quantities. As a result model equations are formulated using integral operators and integral equations, in contrast to classical ones that rely on differential operators and differential equations. This characteristic makes nonlocal models amenable to describe singular and discontinuous physical, social and biological phenomena, see \cite{tadmor,Gilboa-Osher,appl:04,rossi} for applications and analysis of nonlocal equations. Our interest centers on models in peridynamics, a nonlocal reformulation of the basic equations of motion in continuum mechanics, that have shown promising potential in modeling the spontaneous formation of discontinuities in solids.  The  present work studies qualitative properties of solutions to the equilibrium equation in the  linearized bond-based peridynamic model that first appeared in \cite{Silling2000}, with a generalization later appearing in \cite{Silling2007, Silling2010}. To describe the model, a material body occupying a region has undergone the deformation that maps a material point $\bdx$ to $\bdx + {\bf u}(\bdx)$ in a deformed domain. Clearly, the vector field ${\bf u}$ represents the displacement field. Treating the material body as a complex mass-spring system, in peridynamics it is postulated that material points $\bdy$ and $\bdx$ interact through a 
bond vector $ \bdy-\bdx$. Under the uniform small strain theory \cite{Silling2010}, the strain of the bond $\bdy - \bdx$ is given by the
nonlocal linearized strain $
 {\bu(\bx)-\bu(\by)\over |\bx-\by|} \cdot \frac{\bx-\by}{|\bx-\by|}. $ 
A portion of this strain contributes to the volume changing component of the deformation and the
remaining is the shape changing component.  According to the linearized bond-based peridynamic model  \cite{Silling2010},  the balance of forces is given by a system of nonlocal equations
\[
c_{\mathfrak{h}}\int_{B_{\mathfrak{h}}(\bdx)} { \rho(\bx, \by){(\bdx-\bdy) \over |\bdx-\bdy|}\otimes{(\bdx-\bdy)\over |\bdx-\bdy|}} \left({\bf u}(\bdx) - {\bf u}(\bdy)\right)\, \mathrm{d}\bdy = {\bf F}(\bf x) 
\]
 where ${\bf F}(\bf x)$ is a vector of applied forces, $\rho$ is a nonnegative measurable function that represents the strength of interactions between material points $\bdx$ and $\bdy$, and finally the positive number $\mathfrak{h}$, called the {horizon}, determines the extent of interaction. The positive number  $c_\mathfrak{h}$ is a normalizing constant chosen in such a way that for smooth deformations the nonlocal operator converges to a differential operator when $\mathfrak{h}\to 0$. The kernel $\rho(\bx,\by)$ contains properties of the modeled material and typically decreases when $|\bx-\by|$ gets larger. It may depend on material points $\bdx, \bdy$, their relative position $\bdy-\bdx$, or in the case of isotropic materials only on their relative distance $|\bdy-\bdx|$. For general $\rho$, the equation may model heterogeneous and anisotropic materials.  In the above system the left hand side represents  
  the linearized internal force density function due to the deformation $\bdx\mapsto\bdx + {\bf u}(\bdx)$ and  is  a weighted average of the nonlocal linearized strain function associated with the displacement ${\bf u}$. See the papers \cite{Mengesha-DuElasticity, Silling2010} for derivation. See also the papers \cite{Emmrich-Weckner2007,Du-Navier1,Du-ZhouM2AN,Du-Zhou2010} for some mathematical analysis of linearized models.   Under the {\em small strain} regime nonlocal nonlinear peridynamic evolution models have also been studied \cite{Lipton14, Lipton15}. Nonlocal functionals related to the peridynamic model are also studied in \cite{Foss-Radu-2016}. 
  
This work will focus on a system of equations that uses interaction kernels $\rho (\bdx, \bdy)$ which behave like $|\bdx- \bdy|^{-(d+2s)}$ for $|\bdx-\bdy|$ close to zero and infinity. To be precise, we study the system of nonlocal equations of the type formally given by
\begin{equation}\label{maineqn}
\mathbb{L}_{\mathfrak{h}}{\bf u}(\bx) := c_\mathfrak{h}\int_{B_{\mathfrak{h}}(\bdx)} {A(\bdx, \bdy)\over |\bdx-\bdy|^{d+2s}}{{(\bdx-\bdy) \over |\bdx-\bdy|}\otimes{(\bdx-\bdy)\over |\bdx-\bdy|}} \left({\bf u}(\bdx) - {\bf u}(\bdy)\right)\, \mathrm{d}\bdy = {\bf F}(\bf x) 
\end{equation}
where $ d\geq 2$, $0< s<1$ are fixed,  and $A(\bdx, \bdy)$ is a measurable function, which we refer to as the \textit{coefficient}, that is elliptic and symmetric in the sense  
\[\alpha_1\leq A(\bdx, \bdy)\leq \alpha_2,  \quad \quad A(\bdx, \bdy) = A(\bdy, \bdx) \quad \quad \text{ for all $\bdx, \bdy\in \mathbb{R}^{d}$.} \]   
  We will give later the precise notion the operator is defined. 
  
  The goal of this paper is twofold. The first is to establish the higher integrability of a measure of smoothness of weak solutions  to  the peridynamic system of nonlocal  equations given in \eqref{maineqn}.   The notion of weak solution will be be defined in the next section.  
Roughly speaking, we show that there exists an exponent  $p > 2$ such that for any weak solution ${\bf u}$ to  \eqref{maineqn} corresponding to  a rough data $\bF\in L^{r}(\mathbb{R}^{d};\mathbb{R}^{d})$, for $r < 2$ both $\bu$ and $\Upsilon^{s}(\bu) $ are in $L^{p}$ where the function $\Upsilon^{s}({\bf u})$ is a measure of local smoothness of $\bu$ given by 
\begin{equation} \label{DefGamma}
\Upsilon^{s} ({\bf u}) (\bdx) := \left(\intdm{\mathbb{R}^{d}}{ \frac{|{\bf u}(\bdx)-{\bf u}(\bdy)|^{2}}{|\bdx-\bdy|^{d+2s}}}{\bdy} \right)^{1\over 2}. 
\end{equation}
The higher integrability result holds under no additional assumption on the coefficient $A(\bx,\by)$ other than ellipticity and measurability.   
The second goal is to characterize the space of vector fields ${\bu}$ such that both $\bu$ and $\Upsilon^{s}(\bu) $ are in $L^{p}(\mathbb{R}^{d})$. In fact, the characterization we present here utilizes the smaller function $D^{s}(\bu)$ 
that  is given by 
{ \begin{equation}\label{NormDs}
D^{s}({\bf u}) (\bdx) = \left(\int_{\mathbb{R}^{d}}{ \frac{ \left| ({\bf u}(\bdx)-{\bf u}(\bdy))\cdot {(\bdx-\bdy) \over |\bdx-\bdy|} \right|^{2}}{|\bdx-\bdy|^{d+2s}}}{d\by}\right)^{1\over 2}\,.
\end{equation} }
The function $D^{s}({\bf u})$ is the natural measure of  the $s$-differentiability of ${\bf u}$ directly related to the peridynamic energy and it satisfies $ D^{s}({\bf u}) (\bdx) \leq \Upsilon^{s}(\bu)(\bdx)$ for all $\bdx\in \mathbb{R}^{d}$. We will establish that for $\textstyle p\in{\small\left( {2d\over d +2s}, \infty\right)}$, the space of functions in which both ${\bf u}$ and $D^{s}({\bf u})$ are in $L^{p}$ is the standard Bessel potential space $\mathcal{L}^{p}_{s}$, which will  be defined shortly, that is, 
$
\{{\bf u} \in L^{p}(\mathbb{R}^{d};\bbR^{d}): D^{s}(\bu)  \in L^{p}(\bbR^{d})\} = \mathcal{L}^{p}_{s}(\mathbb{R}^{d}). 
$

We should note that one may think of the system \eqref{maineqn} as a fractional analogue of the strongly coupled system of partial differential equations
\begin{equation}\label{local-eqn}
\text{{\bf div}}\, \mathfrak{C}(\bdx) \nabla {\bf u}(\bdx) = {\bf F}(\bdx) \end{equation}
where $\mathfrak{C}(\bdx) $ is a fourth-order tensor of bounded coefficients. Systems of differential equations of the above type are commonly used  in elasticity and  are not necessarily uniformly elliptic but rather satisfy the weaker Legendre-Hadamard ellipticity condition. For a class of coefficients this connection between the nonlocal system \eqref{maineqn} and the local system \eqref{local-eqn} is rigorously justified in \cite{Du-ZhouM2AN,Mengesha-Du-non, Mengesha-DuElasticity} in the event of vanishing nonlocality ($\mathfrak{h}\to 0$). As has already been shown in \cite[Theorem 3]{Mengesha-DuElasticity} via a simple calculation using Taylor expansion, for ${\bf u}\in C_c^{2}(\mathbb{R}^{d};\mathbb{R}^{d})$, fixed $s$, and $A(\bdx, \bdy) = a(\bdx)$, there is a coefficient $\mu(\bdx)$ which is a constant multiple of $a(\bdx)$ such that as $\mathfrak{h}\to 0$
\[
\mathbb{L}_{\mathfrak{h}}{\bf u} \to \mathbb{L}_{0}{\bf u}:=\text{div}(\mu(\bdx)\nabla {\bf u}) + 2\nabla(\mu(\bdx) \textrm{div}\nabla {\bf u})\,.
\]
In light of this connection, the higher integrability result for weak solutions to \eqref{maineqn} is a fractional analogue of Meyers inequality for systems of differential equations. Meyers inequality states that weak solutions of the strongly coupled systems \eqref{local-eqn} with measurable coefficients satisfying the Legendre-Hadamard ellipticity condition and corresponding to highly integrable data are in $W^{1,p}(\mathbb{R}^{d};\mathbb{R}^{d})$  for some $p >2$, see \cite{Meyers1,Meyers}.  The result we obtain is the fractional analogue of this inequality in the Bessel potential spaces.  

In what follows we assume that the horizon $\mathfrak{h}$ is given and fixed. We will not track the dependence of generic constants on the horizon.   In the next section, we introduce notations, give the definition of weak solutions and state the precise statement of the main results. We will also describe how we prove those results.  In Section \ref{Meyers-section} we prove the higher integrability result. We first establish invertibility over a range of Sobolev spaces for the operator $\bbI_d + \bbL$. For a given weak solution $\bu$ to the nonlocal elliptic system \eqref{maineqn}, we will use this invertibility to prove higher integrability of $D^{s}({\bu})$. In Section \ref{Characterization-Section} we prove  the characterization of the potential spaces.
\section{Statement of main results}\label{sec:mainresult}
To state the main results of the paper we first introduce the quadratic bilinear form associated with the operator $\mathbb{L}_{\mathfrak{h}}$ in \eqref{maineqn}:   
{\small \[
\mathcal{E}_{\mathfrak{h}}({\bf u}, {\bf v}) = {c_\mathfrak{h}\over 2}\intdm{\mathbb{R}^{d}}{\intdm{B_{\mathfrak{h}}(\bdx)}{{A(\bdx, \bdy)\over |\bdx-\bdy|^{d+2s}} ({\bf u}(\bdx)-{\bf u}(\bdy))\cdot {(\bdx-\bdy) \over |\bdx-\bdy|} ({\bf v}(\bdx)-{\bf v}(\bdy))\cdot {(\bdx-\bdy) \over |\bdx-\bdy|}}{\bdy}}{\bdx}\,.
\]}
We use the notation $\mathcal{E}$ to represent the bilinear form when both integrals are on $\bbR^{d}$, that is formally, the horizon ${\mathfrak{h}} = \infty$. In this case we take $c_\infty = 1$, and write the operator as $\bbL$. 
For ${\bf u}\in L^{2}(\mathbb{R}^{d};\mathbb{R}^{d})$, the vector valued map  $\mathbb{L}_{\mathfrak{h}}{\bf u}$ is a vector of distributions acting on test functions ${\bfs \phi}\in C^{\infty}_c(\mathbb{R}^{d};\mathbb{R}^{d})$   as
\[
\langle\mathbb{L}_{\mathfrak{h}}{\bf u}, {\bfs \phi}\rangle : = \mathcal{E}_{\mathfrak{h}}({\bf u}, {\bfs \phi}). 
\]
It is now clear that if ${\bf u}$ is in the energy space $\{{\bf u}\in L^{2}: \mathcal{E}_{\mathfrak{h}}({\bf u},{\bf u}) < \infty\}$, then $\mathbb{L}_{\mathfrak{h}}{\bf u}$ is in its dual space.   
We will shortly characterize that the energy $\mathcal{E}_{\mathfrak{h}}({\bf u},{\bf u})$ is finite if and only if ${\bf u}\in W^{s,2}(\mathbb{R}^{d};\mathbb{R}^{d})$. This assertion follows from the lower and upper bounds of the coefficient $A(\bdx, \bdy)$ and using the equivalence of spaces that is recently proved in \cite{Mengesha-Hardy}.  The notion of weak solution to \eqref{maineqn}  is standard and is given in terms of the bilinear form $\mathcal{E}$. 
\begin{definition}
Given $\mathbf{F}\in W^{-s,2}(\mathbb{R}^{d};\mathbb{R}^{d})$, we say that ${\bf u}\in W^{s,2}(\mathbb{R}^{d};\mathbb{R}^{d})$ is a \textit{weak solution to} \eqref{maineqn} if 
\begin{equation}\label{defnweaksoln}
\mathcal{E}_{\mathfrak{h}}({\bf u},{\bfs \phi})= \langle{\bf F}, {\bfs \phi}\rangle
\end{equation}
for any ${\bfs \phi}\in W^{s,2}(\mathbb{R}^{d};\mathbb{R}^{d})$. In the above $\langle\cdot, \cdot\rangle$ denotes the duality pairing. 
\end{definition}
Existence and uniqueness of weak solutions to this particular system will be demonstrated shortly, see also 
the recent work in \cite{Kassmann-Mengesha-Scott} when complementary boundary conditions is imposed on the solution. However, the focus of this paper will be the issue of regularity of solutions. We seek to address the following question: If the data is regular, how regular is the solution? In particular,  we are interested in data ${\bf F}$ coming from a class of low integrability. We use the following standard notations for the fractional Sobolev exponents and its H\"older conjugate (denoted by the prime notation)  
\[
2^{\ast_{s}} = {2d\over d-2s}\,, \quand \left(2^{\ast_{s}}\right)' := 2_{\ast_{s}} = {2d \over d+2s}\,.   
\] 

Our result on the Sobolev regularity of solutions to the system \eqref{maineqn} states that any weak solution ${\bf u}$ to the nonlocal system with Poisson data $\bF$ in the described class in fact belongs to the Bessel potential space $\mathcal{L}^{p}_{s}(\mathbb{R}^{d})$ for some $p > 2$.  The space  $\mathcal{L}^{p}_{s}(\mathbb{R}^{d})$ will be defined shortly but we note that  when $p > 2$, the Bessel potential space is finer than the Besov space, i.e.\  $\mathcal{L}^{p}_{s}(\bbR^{d}) \subset W^{s,p}(\mathbb{R}^{d};\mathbb{R}^{d})$.   For precision, we state the result in the following theorem. 
\begin{theorem}\label{Meyers-type}
Suppose that $0<\alpha_1\leq \alpha_2<\infty$, $0<s<1$, $d\geq 2$, $\delta_0 > 0$ and $0 < \mathfrak{h} \leq \infty$ are all given. Let $A(\bdx, \bdy)$ be symmetric and  measurable with $\alpha_1\leq A(\bdx, \bdy)\leq \alpha_2$ and  ${\bf u}\in W^{s,2}(\mathbb{R}^{d};\mathbb{R}^{d})$ be any weak solution in the sense of  \eqref{defnweaksoln} corresponding to a given ${\bf F} \in L^{2_{*_{s}}  + \delta_0}(\mathbb{R}^{d}; \mathbb{R}^{d}) \cap L^{2_{*_s}}(\mathbb{R}^{d}; \bbR^d)$. Then there exists $\veps > 0$ depending only on $\alpha_1$, $\alpha_2$, $d$ and $s$ with the following property: If $p \in [2,\infty)$ satisfies
\begin{equation}\label{ExponentConstraints}
p \in [2,2^{*_s})\,, \qquad p-2 < \veps\,, \qquad \frac{dp}{(2 - \frac{p}{2})d +sp} - \frac{2d}{d+2s} < \delta_0\,,
\end{equation}
then $\bu \in \cL^s_p(\bbR^d;\bbR^d)$. Moreover, there exists a constant $C > 0$ such that
\begin{equation}
\Vnorm{\bu}_{\cL^s_p(\bbR^d)} \leq C \Big( \Vnorm{\bF}_{L^{2_{*_s}+\delta_0}(\bbR^d)} + \Vnorm{\bF}_{L^{2_{*_s}}(\bbR^d)} + \Vnorm{\bu}_{W^{s,2}(\bbR^d)} \Big)\,.
\end{equation}
The constant $C$ depends only on $\alpha_1$, $\alpha_2$, $d$, $p$, $\delta_0$ and $\mathfrak{h}$.
\end{theorem}
The main implication  of the regularity result is that, with no additional smoothness condition on the coefficient $A(\bdx, \bdy)$, a  weak solution satisfying  \eqref{defnweaksoln} has a higher integrable fractional ``s-derivative".   For scalar equations, this type of nonlocal analogue of Meyers inequality was obtained 
 in the recent work \cite{Bass-Ren}\footnote{The argument presented in the paper appears to have a gap that we are unable to fix as of the writing of this manuscript. }.  It has been shown in several subsequent works that solutions to fractional equations surprisingly also improve in \textit{differentiability}, again with no additional smoothness conditions on the coefficients. Some of these approaches use reverse H\"older inequalities and nonlocal Gehring-type lemmas as in \cite{Kuusi-M-S}, or prove the result via a commutator estimate as in \cite{Schikorra} and  in  \cite{Auscher} using a functional analytic approach.    The approaches in \cite{Kuusi-M-S, Schikorra} are local in nature and  application of appropriate embedding estimates in their work show improved local differentiability will lead to improved local integrability. The approach in \cite{Auscher} is rather robust and implies local and global regularity results.     We should mention that the results in \cite{Kuusi-M-S, Schikorra} hold for vectorial systems of fractional equations,   even for coupled systems of fractional equations such as the vectorial fractional $p$-Laplacian, which are uniformly elliptic in their natural energy space.  However, none of these works apply directly to the system of coupled nonlocal equations under consideration in this paper which has a weaker ellipticity property in the spirit of Legendre-Hadamard.  The proof of Theorem \ref{Meyers-type} we present in this paper follows the argument presented in  \cite{Auscher} for nonlocal scalar equations extending its applicability to the system of coupled nonlocal equations and for rough data.  

We also remark that for $\delta_0$ small, the exponent $2_{\ast_{s}} +\delta_{0}$ can be smaller than $2$, and the higher potential space estimate holds true for rough data $\bF \in L^{r } (\bbR^d, \bbR^d)$ and $r < 2$. We note that there are vector functions ${\bf F} $ that are in $L^{2_{*_{s}}  + \delta_0}(\mathbb{R}^{d}; \mathbb{R}^{d}) \cap L^{2_{*_s}}(\mathbb{R}^{d}; \bbR^d)$ but not in $L^{2}(\mathbb{R}^{d};\mathbb{R}^{d})$. Indeed, for $\delta_0$ sufficiently small and $\textstyle \tau = \frac{d \delta_0}{2_{*_s}(2_{*_s}+\delta_0)}$ it is straightforward to verify via Sobolev embedding that vector fields in $W^{\tau, 2_{*_s}}(\bbR^d;\bbR^d)$ will belong to the former class but not necessarily the latter. While the higher integrability result remains true  in the event $\bF\in L^{2} (\bbR^d, \bbR^d)$,  this case must be treated differently. In fact for nonlocal scalar equations an argument is presented in \cite[Corollary 4.9]{Auscher} that is also applicable to our case.

The function $D^{s}({\bf u}) $ is related to the energy space $W^{s,2}(\mathbb{R}^{d};\mathbb{R}^{d})$. A vector field ${\bf u}\in L^{2}(\mathbb{R}^{d}; \mathbb{R}^{d})$ will belong to $W^{s,2}(\mathbb{R}^{d};\mathbb{R}^{d})$ if and only if $D^{s}({\bf u})\in L^{2}(\mathbb{R}^{d})$. This can easily be seen using the fractional Korn-type inequality proved in \cite[Theorem 1.1]{Scott-Mengesha} that states the space 
\begin{equation}\label{Space-X}
 \mathcal{X}_{q}^{s}(\bbR^{d}) = \left\{ {\bf v}\in L^{q}(\bbR^{d}; \mathbb{R}^{d})\, : \, \intdm{\bbR^{d}}{ \intdm{\bbR^{d}}{ \frac{ \left| ({\bf v}(\bdx) - {\bf v} ({\bdy})) \cdot \frac{(\bdx -\bdy)}{|\bdx - \bdy|}  \right|^q}{|\bdx-\bdy|^{d+sq}}}{\bdy}}{\bdx} < \infty \right\}
 \end{equation}
is equivalent with the standard fractional Sobolev space $W^{s,q}(\mathbb{R}^{d};\mathbb{R}^{d})$, for any $s\in (0, 1)$ and $q\in (1, \infty)$; see also \cite{Mengesha-Hardy}.

 The higher integrability of $D^{s}({\bf u})$ will be shown to be equivalent with the statement that ${\bf u}$ is in a higher-integrable potential space. Establishing this equivalence is the second goal of the paper. To this end, we first introduce the space of Bessel potentials, $\cL^{q}_{s}$. For $0 < s < 1$, $1 < q < \infty$ 
\begin{equation}
\cL^{q}_{s}(\bbR^d) := \left\lbrace \bff \in \mathcal{S}'(\bbR^d;\mathbb{R}^{d}) \, : \, \left(\left(1 +4\pi^2 |\bfxi|^2\right)^{s/2} \widehat{\bff}\right)^\vee\in L^q(\bbR^d;\mathbb{R}^{d}) \right\rbrace\,,
\end{equation}
where $\mathcal{S}'(\bbR^d;\mathbb{R}^{d})$ is the space of vector-valued tempered  distributions, and the Fourier transform $\widehat{\bh}$ is defined for smooth vector fields $\bh$ as 
\[
\widehat{\bh}(\bfxi) = \intdm{\mathbb{R}^{d}}{ e^{-\imath2\pi\bfxi\cdot\bdx} \bh(\bdx)}{\bdx}\,.
\]
The notation $\cdot ^{\vee}$ denotes the inverse Fourier transform. We will retain the same Fourier transform notations even  for distributions.  
The norm in $\cL^{q}_{s}(\bbR^d)$ is given by 
\begin{equation}
\Vnorm{\bff}_{\cL^{q}_{s}(\bbR^d)} := \left\|\left(\left(1 +4\pi^2 |\bfxi|^2\right)^{s/2} \widehat{\bff}\right)^\vee\right\|_{ L^q(\bbR^d) }. 
\end{equation}

Our second main result is the characterization of functions in the potential space $\cL^{q}_{s}$ in terms of the integrability of the smaller function $D^{s}(\bff )$. To that end, let us introduce the function space $\mathcal{D}^{p}_{s}(\mathbb{R}^{d})$ which is the closure of the Schwarz space $\mathcal{S}(\bbR^d;\bbR^d)$ in the norm $\|{\bf u}\|_{L^{p}(\bbR^d)} + \|D^{s}({\bf u})\|_{L^{p}(\bbR^d)}$.  
\begin{theorem}\label{Character-Potential}
Let $2_{*_{s}}\leq p < \infty$ and $0 < s < 1$. Then  $\bff \in \cL^{p}_{s}(\bbR^d)$ if and only if $\bff  \in \mathcal{D}^{p}_{s}(\mathbb{R}^{d})$. Moreover, there exist positive constants $C_1$ and $C_2$ such that for any ${\bf f}\in L^{p}(\mathbb{R}^{d};\mathbb{R}^{d})$, 
\[
C_1 \| \bff\|_{\cL^{p}_{s}(\bbR^d)} \leq \|{\bf f}\|_{L^{p}(\bbR^d)} + \|D^{s}({\bf u}) \|_{L^{p}(\bbR^d)} \leq C_2 \| \bff\|_{\cL^{p}_{s}(\bbR^d)}\,.
\]
\end{theorem}

We emphasize that this characterization is in the same spirit as classical characterizations of vector fields in the Bessel potential space. One such characterization is given by Stein  in \cite[Theorem 1]{Stein-Bessel} or \cite[Chapter V]{Stein} in terms of  smallness of the difference $\bff(\bdx + \bdy) - \bff(\bdx)$ via the Marcinkiewicz integral $\Upsilon^{s}(\bff)$. The result states that for $0 < s < 1$ and ${2_{*_{s}}} < p < \infty$, $\bff \in \cL^{p}_{s}(\bbR^d;\bbR^d)$ if and only if $\bff \in L^p(\bbR^d;\bbR^d)$ and $\Upsilon^{s}(\bff) \in L^p(\bbR^d)$.  A finer characterization that uses means over balls is also given in \cite{Strichartz} that is valid for the full range of $p\in (1, \infty)$. Various other characterizations have also been explored throughout the literature. However, the essence of Theorem \ref{Character-Potential} lies in the fact that  to determine if $\bff$ is a  Bessel potential of an $L^p$ vector field  we do not need the smallness of a measure of $\bff(\bdx + \bdy) - \bff(\bdx)$ but rather a measure of the quantity $[\bff(\bdx + \bdy) - \bff(\bdx)]\cdot {\bdy\over |\bdy|}$ via another Marcinkiewicz-type integral of $\bff $, $D^{s}({\bff})$.  Since for any $\bdx\in \bbR^{d}$ the pointwise estimate $D^{s}({\bff})(\bx) \leq \Upsilon^{s}(\bff)(\bdx)$ holds, the right-hand side inequality in Theorem \ref{Character-Potential} follows from the characterization in \cite[Theorem 1]{Stein-Bessel}. However, the left-hand side inequality is a refined version of  \cite[Theorem 1]{Stein-Bessel} since we are estimating $\| \bff\|_{\cL^{p}_{s}(\bbR^d)}$ in terms of the smaller function $D^{s}(\bff)$ in place of $\Upsilon^{s}(\bff)$. We will in fact show that the left hand side inequality in the theorem is valid for $p\in (1, \infty)$.

Our proof of   \autoref{Character-Potential} follows the steps presented in the proof of \cite[Theorem 1]{Stein-Bessel}.  We first develop the necessary technical tools that allow us to relate the Marcinkiewicz-type integral $D^{s}({\bff})$ with the potential function of $\bff$.  It turns out this is possible by introducing a Poisson-type integral of ${\bff}$ and a corresponding Littlewood-Paley $g$-function.  We will show that, in parallel with classical results, this new $g$-function can be used to characterize $L^{p}$ norms of vector fields.

\section{Meyers-type higher integrability result for coupled nonlocal systems}\label{Meyers-section}
In this section we prove \autoref{Meyers-type}. As we have discussed in the introduction our proof of  \autoref{Meyers-type} follows the argument presented in \cite{Auscher} in obtaining a higher integrability result for scalar nonlocal elliptic equations. We include a reproduction of their methods verifying its applicability for the strongly coupled system of equations under consideration. The first result we prove here is the higher integrability of $D^{s}({\bf u})$ 
using the analytic perturbation result. 
\begin{theorem}\label{Meyers-type-first step}
Suppose that $0<\alpha_1\leq \alpha_2<\infty$, $0<s<1$, $d\geq 2$, $0<\mathfrak{h}\leq \infty$, and $\delta_0 >0$ are given. Then there exists $\veps > 0$ depending only on $\alpha_1$, $\alpha_2$, $d$ and $s$ with the following property: if $p > 2$ satisfies the conditions \eqref{ExponentConstraints}, then there exists a constant $C > 0$ such that for any  $A(\bdx, \bdy)$ symmetric, measurable with $\alpha_1\leq A(\bdx, \bdy)\leq \alpha_2$ and any weak solution ${\bf u}\in W^{s,2}(\mathbb{R}^{d};\mathbb{R}^{d})$ in the sense of \eqref{defnweaksoln} corresponding to 
 ${\bf F} \in L^{2_{*_{s}}  + \delta_0}(\mathbb{R}^{d}; \mathbb{R}^{d}) \cap L^{2_{*_{s}}}(\mathbb{R}^{d}; \mathbb{R}^{d})$, 
we have
\[
\|{\bf u}\|_{L^{p}(\bbR^d)} + \|\Upsilon^{s}({\bf u})\|_{L^{p}(\bbR^d)} \leq C\left( \| {\bf F} \|_{L^{{2_{\ast_{s}}} + \delta_0}(\bbR^d)} + \| {\bf F} \|_{L^{{2_{\ast_{s}}}}(\bbR^d)} + \|{\bf u}\|_{W^{s,2}(\bbR^d)} \right)\,. 
\]
The constant $C$ depends only on $\alpha_1$, $\alpha_2$, $d$, $p$, $\delta_0$ and $\mathfrak{h}$.
\end{theorem}
By the log-convexity of $L^p$ norms the data ${\bf F}$ belongs to $L^r(\bbR^d;\bbR^d)$ for any $r \in (2_{*_s}, 2_{*_s} + \delta_0)$. In fact, we will see that one can replace the two $L^p$ norms of $\bF$ on the right-hand side of the regularity estimate with the single term $\Vnorm{\bF}_{L^r}$, where $\textstyle r = \frac{dp}{(2-\frac{p}{2})d+sp}$.

We notice that  \autoref{Meyers-type}  now follows as a corollary of this theorem. Indeed, applying the characterization theorem (\autoref{Character-Potential}) we see that ${\bf u}\in \mathcal{L}^{p}_{s}(\mathbb{R}^{d})$. Moreover, 
by applying the well-known embedding \cite[Chapter V]{Stein} of the spaces $ \mathcal{L}^{p}_{s}(\mathbb{R}^{d}) \subset W^{s, p}(\bbR^{d};\bbR^{d})$ for $p > 2$, \autoref{Meyers-type} implies a higher Besov space regularity result for solutions.

\subsection{Key Inequalities}
Before we prove \autoref{Meyers-type-first step}, we first state two inequalities. The first inequality is the classical fractional Sobolev embedding. The second inequality we call a ``Fractional Korn's Inequality," and establishes the equivalence of the space $\mathcal{X}^{s}_{q}(\mathbb{R}^{d})$ introduced in \eqref{Space-X} with the standard fractional Sobolev space $W^{s,q}(\mathbb{R}^{d};\mathbb{R}^{d})$ for any $s \in (0, 1)$ and $q\in (1, \infty)$. This inequality will be used to prove an invertibility result for the operator $\varpi\bbI_d + \bbL$ for some $\varpi > 0$.

\begin{lemma}[\hspace{-0.2pt}\cite{Auscher}, Lemma 4.2]\label{lma:SobolevEmbedding}
Suppose $s \in (0,1)$ and $q \in [1,\infty)$. Suppose also that $sq < d$. Then
\begin{equation}
\Vnorm{\bu}_{L^{q^{*_s}}(\bbR^d)} \leq C [\bu]_{W^{s,q}(\bbR^d)}\,.
\end{equation}
In particular, $W^{s,q}(\bbR^d;\bbR^d) \subset L^{q^{*_s}}(\bbR^d;\bbR^d)$ and $W^{s,q_{*_s}}(\bbR^d;\bbR^d) \subset L^{q}(\bbR^d;\bbR^d)$.
\end{lemma}
 
\begin{theorem}[\hspace{-0.2pt}\cite{Scott-Mengesha} {\bf Fractional Korn's inequality}]\label{theorem-korn}
For any $s\in (0, 1)$ and $1 < q < \infty$,  
$$
\mathcal{X}^{s}_{q}(\mathbb{R}^{d}) = W^{s, q}(\mathbb{R}^{d};\mathbb{R}^{d})\,.
$$
Moreover, there exists a universal constant $\kappa = \kappa(d, q, s)$ such that for all vector fields ${\bf f}\in W^{s, q}(\mathbb{R}^{d};\mathbb{R}^{d})$ we have 
\begin{equation}\label{frackorn}
[{\bf f}]_{\mathcal{X}^{s}_{q}(\bbR^d)} \leq  [{\bf f}]_{W^{s,q}(\bbR^d)} \leq \kappa [{\bf f}]_{\mathcal{X}^{s}_{q}(\bbR^d)}\,.
\end{equation}
\end{theorem}
For a given bounded domain $\Omega,$ the function space $\mathcal{X}^{s}_{q}(\Omega)$ is defined in the same way as $\mathcal{X}^{s}_{q}(\mathbb{R}^{d})$ but the functions are defined in $\Omega$ and the integrations in the semi-norm are on $\Omega$. 
We call the above theorem a ``Fractional Korn's inequality" 
because it has been shown in \cite{Mengesha} that
$$
\lim\limits_{s \to 1^-} (1-s) \left[ \bff \right]_{\cX^s_q(\Omega)} = \Vnorm{(\grad \bff)_{Sym}}_{L^q(\Omega)} \quad \text{ for every } \bff \in W^{1,q}_{Sym}(\Omega;\bbR^d)\,.
$$
This association suggests that $\mathcal{X}_{q}^{s}(\Omega)$ is the fractional analogue  of  $W^{1,q}_{Sym}(\Omega;\mathbb{R}^{d})$, 
which in turn is known to coincide with $W^{1, q}(\Omega;\mathbb{R}^{d})$ via the classical Korn's inequality.  \autoref{theorem-korn} extends  this equivalence to fractional spaces defined on the whole space.

\subsection{Higher integrability when the horizon is infinite} 
In this subsection we will prove \autoref{Meyers-type-first step} 
when the horizon is infinite. In this case we recall that the bilinear form we use is $\mathcal{E}$, the operator is written $\bbL$, and $c_{\infty}=1$.
We will show that there exists a constant $\varpi > 0$ such that the operator $\textstyle \varpi \bbI_d + \bbL$ from $W^{s,2}(\bbR^d;\bbR^d)$ to $[W^{s,2}(\bbR^d;\bbR^d)]^*$ is invertible, where we use the notation $Y^*$ for the dual space of $Y$. We then use a perturbation lemma of Schneiberg reproduced in \cite{Auscher} to deduce that in fact $\textstyle \varpi \bbI_d + \bbL$ is invertible on a range of ``nearby" fractional Sobolev spaces forming a complex interpolation scale. This is the key step in proving the higher potential space regularity result for weak solutions in the sense of \eqref{defnweaksoln}.
 
\begin{lemma}\label{lma:InvertiblityofI+L}
Let $0<\alpha_1\leq A(\bx,\by) \leq \alpha_2<\infty$, $0<s<1$ and $d\geq 2$. Let $t$, $t' \in (0,1)$ and $p$, $p' \in (1,\infty)$ satisfy $t + t' = 2s$ and $\textstyle p' = \frac{p}{p-1}$. Then there exists $\veps > 0$ such that if $\textstyle \frac{1}{2} - \frac{1}{p} < \veps$ and $t-s < \veps$, then
\begin{equation*}
\varpi\bbI_d + \bbL : W^{t,p}(\bbR^d;\bbR^d) \to [W^{t',p'}(\bbR^d;\bbR^d)]^*
\end{equation*}
is invertible, where $\varpi = \frac{\alpha_1}{\kappa}$ and $\kappa$ is as in \eqref{frackorn} corresponding to $q=2$. Any inverse agrees with the inverse obtained for $t=s$ and $p=2$ on their common domains of definition. Both $\veps$ and the norms of the inverses depend only on $d$, $s$, and the ellipticity constants $\alpha_1$ and $\alpha_2$.
\end{lemma}

\begin{proof}
The objective here is to show that the perturbation lemma of Schneiberg (see \cite[Theorem A.1]{Auscher}) can be applied to the operator $\textstyle \varpi  \bbI_d + \bbL$.
We will prove the theorem in three steps.

\noindent
{\bf Step I}: We show that $\textstyle \varpi  \bbI_d + \bbL : W^{s,2}(\bbR^d;\bbR^d) \to [W^{s,2}(\bbR^d;\bbR^d)]^*$ is bounded and invertible. 
To verify the boundedness of the operator, it suffices to check if $\bbL$ is a bounded operator $ W^{s,2}(\bbR^d;\bbR^d) $ to its dual space. But this is a consequence of the definition of $\bbL$ and H\"older's inequality. Indeed, it follows from the fact that for any ${\bf u}, {\bf v}\in W^{s,2}(\bbR^d;\bbR^d)$\[
|\Vint{\bbL \bu, \bv}| = |\cE(\bu,\bv)| \leq \alpha_2 [\bu]_{\cX^{s,2}(\bbR^d)} [\bv]_{\cX^{s,2}(\bbR^d)} \leq \alpha_2 \Vnorm{\bu}_{W^{s,2}(\bbR^d)} \Vnorm{\bv}_{W^{s,2}(\bbR^d)}. 
\]
To see that $\varpi \bbI_d + \bbL$ is invertible, note that by the ellipticity assumptions on $A$ and by the fractional Korn's inequality Theorem \ref{theorem-korn} applied to $q=2$, we have 
and 
\[
|\Vint{\bbL \bu, \bu}| \geq \alpha_1 [\bu]_{\cX^{s,2}(\bbR^d)}^2 \geq \frac{\alpha_1}{\kappa} [\bu]_{W^{s,2}(\bbR^d)}^2\, = \varpi [\bu]_{W^{s,2}(\bbR^d)}^2, 
\]
implying that $\textstyle \Vint{\varpi  \bu + \bbL \bu, \bu} \geq \varpi \Vnorm{\bu}_{W^{s,2}(\bbR^d)}^2$. Invertibility of the operator $\varpi  \bu + \bbL$  is now a consequence of the well-known Lax-Milgram theorem. 

\noindent{\bf Step II}: The operator $\textstyle \varpi  \bbI_d + \bbL$ extends from $C^{\infty}_c(\bbR^d;\bbR^d)$ by density to a bounded operator $W^{t,p}(\bbR^d;\bbR^d) \to [W^{t',p'}(\bbR^d;\bbR^d)]^*$, also denoted by $\textstyle \varpi  \bbI_d  + \bbL$. 
To see this for any $\bu \in W^{t,p}(\bbR^d;\bbR^d)$ and any $\bv \in W^{t',p'}(\bbR^d;\bbR^d)$, by H\"older's inequality
\begin{equation*}
\begin{split}
& \left| \Vint{\varpi  \bu + \bbL \bu, \bv} \right| \leq \varpi \Vnorm{\bu}_{L^p(\bbR^d)} \Vnorm{\bv}_{L^{p'}(\bbR^d)} \\
&\qquad + \left|  \intdm{\bbR^d} { \intdm{\bbR^d}{A(\bx,\by)\frac{\big( \bu(\bx)-\bu(\by) \big) \cdot \frac{\bx-\by}{|\bx-\by|}}{|\bx-\by|^{\frac{d}{p}+t}} \cdot \frac{\big( \bv(\bx)-\bv(\by) \big) \cdot \frac{\bx-\by}{|\bx-\by|}}{|\bx-\by|^{\frac{d}{p'}+t'}}}{\by}}{\bx} \right| \\
& \qquad \qquad \leq \varpi  \Vnorm{\bu}_{L^p(\bbR^d)} \Vnorm{\bv}_{L^{p'}(\bbR^d)} + \alpha_2 [\bu]_{W^{t,p}(\bbR^d)} [\bv]_{W^{t',p'}(\bbR^d)}\,.
\end{split}
\end{equation*}
Since $C^{\infty}_c(\bbR^d;\bbR^d)$ is dense in all fractional Sobolev spaces involved, it follows then that the operator $\textstyle\varpi \bbI_d + \bbL : W^{s,2}(\bbR^d;\bbR^d) \to [W^{s,2}(\bbR^d;\bbR^d)]^{*}$ extends to a bounded operator $W^{t,p}(\bbR^d;\bbR^d) \to [W^{t',p'}(\bbR^d;\bbR^d)]^*$.

\noindent{\bf Step III}: The remaining argument is essentially the same as the one presented in \cite{Auscher} with the modification that it now applies to fractional Sobolev spaces of vector fields. We include it here for completeness. We begin by displaying the complex interpolation scale of the fractional Sobolev spaces of vector fields, and then use the two steps we proved above to verify the assumptions of Schneiberg's lemma \cite[Theorem A.1]{Auscher}.
Denoting the scale of complex interpolation spaces between two Banach spaces $X_0$ and $X_1$ by $[X_0,X_1]_{\theta}$ for $\theta \in (0,1)$, it is well-known that
\begin{equation*}
\big[ W^{t_0,p_0}(\bbR^d;\bbR^d),W^{t_1,p_1}(\bbR^d;\bbR^d) \big]_{\theta} = W^{t,p}(\bbR^d;\bbR^d)
\end{equation*}
for $t_0$, $t_1 \in (0,1)$ and $p_0$, $p_1 \in (1,\infty)$ with $t$ and $p$ given by
\begin{equation*}
\frac{1}{p} = \frac{1-\theta}{p_0} + \frac{\theta}{p_1}\,, \qquad t = (1-\theta)t_0 + \theta t_1\,.
\end{equation*}
Now consider the spaces $W^{t,p}(\bbR^d;\bbR^d)$ as points in the $(t,1/p)$-plane. Choose $(t_0,{1\over p_0})$, $(t_1,1/p_1)$ so that $(s,1/2)$ lies in the line segment joining them, i.e. so that there exists $\theta^* \in (0,1)$ such that
\begin{equation*}
\big[ W^{t_0,p_0}(\bbR^d;\bbR^d),W^{t_1,p_1}(\bbR^d;\bbR^d) \big]_{\theta^*} = W^{s,2}(\bbR^d;\bbR^d)\,.
\end{equation*}
By definition of the (anti)dual exponents $t_0'$, $t_1'$, $p_0'$, $p_1'$,
\begin{equation*}
\Big[ [W^{t_0',p_0'}(\bbR^d;\bbR^d)]^*,[W^{t_1',p_1'}(\bbR^d;\bbR^d)]^* \Big]_{\theta^*} = [W^{s,2}(\bbR^d;\bbR^d)]^*
\end{equation*}
for the same $\theta^*$. Note also that the spaces $W^{t_0',p_0'}(\bbR^d;\bbR^d)$ and $W^{t_1',p_1'}(\bbR^d;\bbR^d)$ lie on the same line segment defined above.
Thus, Steps I and II in tandem state that $\varpi \bbI_d + \bbL$ is a bounded linear operator from the complex interpolation scale $\big[ W^{t_0,p_0}(\bbR^d;\bbR^d),W^{t_1,p_1}(\bbR^d;\bbR^d) \big]_{\theta}$ to $\Big[ [W^{t_0',p_0'}(\bbR^d;\bbR^d)]^*,[W^{t_1',p_1'}(\bbR^d;\bbR^d)]^* \Big]_{\theta^*}$ for any  $\theta \in (0,1)$, and it is invertible for some $\theta^* \in (0,1)$.
We are now in a position to invoke the quantitative Schneiberg lemma, which states that invertibility at the interior point $(s,1/2)$ of this line segment implies invertibility on some open interval of the line segment containing the point $(s,1/2)$. The width of the line segment depends on the coercivity bound obtained in Step I and the continuity bounds in Step II. The inverses obtained coincide with the inverse obtained at the point $(s,1/2)$ on their common domains of definition. Finally, on every line segment through $(s,1/2)$ we can choose an open interval of the same width containing the point, and they sum up to a two-dimensional $\veps$-neighborhood of $(s,1/2)$ in the $(t,1/p)$-plane, as desired.
\end{proof}

\begin{proof}[Proof of \autoref{Meyers-type-first step} when $\mathfrak{h} = \infty$] Let  ${\bf u}\in W^{s,2}(\mathbb{R}^{d};\mathbb{R}^{d})$ be a weak solution in the sense of \eqref{defnweaksoln} corresponding to 
 ${\bf F} \in L^{2_{*_{s}}  + \delta_0}(\mathbb{R}^{d}; \mathbb{R}^{d}) \cap L^{2_{*_{s}}}(\mathbb{R}^{d}; \mathbb{R}^{d})$.  Let $\veps > 0$ be as in Lemma \ref{lma:InvertiblityofI+L}, and let $p$ satisfy \eqref{ExponentConstraints}. Define $\textstyle t = \frac{d}{2} - \frac{d}{p} + s$, and define $t'$, $p'$ as in Lemma \ref{lma:InvertiblityofI+L}.  We prove the theorem by showing that the solution ${\bf u}$ lives in the space ${\bf u}\in W^{t, p}(\mathbb{R}^{d};\mathbb{R}^{d})$ with appropriate bound, and that the function 
$\Vnorm{\Upsilon^s(\bu)}_{L^p(\bbR^d)} \leq C \Vnorm{\bu}_{W^{t,p}(\bbR^d)}\,,$
for some universal constant $C$.  The latter is proved for scalar functions in \cite{Auscher} and its extension to vector fields follows easily. 
To demonstrate the former we will use Lemma \ref{lma:InvertiblityofI+L}  and prove the estimate 
\begin{equation}\label{eq:HigherDiffBound}
\Vnorm{\bu}_{W^{t,p}(\bbR^d)} \leq C \Big( \Vnorm{\bF}_{L^{2_{*_s}+\delta_0}(\bbR^d)} + \Vnorm{\bF}_{L^{2_{*_s}}(\bbR^d)} + \Vnorm{\bu}_{W^{s,2}(\bbR^d)} \Big)\,.
\end{equation}
To that end, since ${\bf u}$ is a weak solution to the coupled system in the sense of \eqref{defnweaksoln}, we write \eqref{maineqn} in the form
\begin{equation}\label{eq:I+LWithData}
\left( \varpi \bbI_d + \bbL \right) \bu = \bF + \varpi \bu\,.
\end{equation}
in the dual space where $\varpi$ is as in Lemma  \ref{lma:InvertiblityofI+L} . 
Note that $t-s \in (0,\veps)$ by choice of $t$ and by  \eqref{ExponentConstraints}. 
Thus by Lemma \ref{lma:InvertiblityofI+L}, the operator $\textstyle \varpi \bbI_d + \bbL : W^{t,p}(\bbR^d;\bbR^d) \to [W^{t',p'}(\bbR^d;\bbR^d)]^*$ is invertible, and by \eqref{eq:I+LWithData}
\begin{equation}\label{eq:HigherDiffBound1}
\Vnorm{\bu}_{W^{t,p}(\bbR^d)} \leq C \Vnorm{\bF + \varpi\bu}_{[W^{t',p'}(\bbR^d)]^*}\,,
\end{equation}
provided the right-hand side is finite. We now show that this is in fact the case. 
Note that $t'p' < 2s < 2 \leq d$, and therefore by the fractional Sobolev embedding theorem and by choice of $t$,
\[
\Vnorm{\bF}_{W^{t',p'}(\bbR^d)]^*} \leq C \Vnorm{\bF}_{L^r(\bbR^d)}\,, \qquad r = \frac{dp}{(2-\frac{p}{2})d + sp}\,.
\]
The exponent $r$ satisfies $r > 2_{*_s}$, with $r \to 2_{*_s}$ as $p \to 2$. Hence, for $p$ satisfying \eqref{ExponentConstraints} we have $r \in (2_{*_s},2_{*_s}+\delta_0)$. By the log-convexity of $L^p$ norms there exists $\beta \in (0,1)$ such that $\Vnorm{\bF}_{L^r(\bbR^d)} \leq \Vnorm{\bF}_{L^{2_{*_s}+\delta_0}(\bbR^d)}^{\beta} \Vnorm{\bF}^{1-\beta}_{L^{2_{*_s}}(\bbR^d)}$, and using the inequality $a^{\beta}b^{1-\beta} \leq a+b$,
\begin{equation*}
\Vnorm{\bF}_{[W^{t',p'}(\bbR^d)]^*} \leq C \Big( \Vnorm{\bF}_{L^{2_{*_s}+\delta_0}(\bbR^d)} + \Vnorm{\bF}_{L^{2_{*_s}}(\bbR^d)} \Big)\,.
\end{equation*}
To control the $(W^{t',p'})^*$ norm of $\bu$ we follow a similar program. Treating $\bu$ as a distribution acting on functions in $W^{t',p'}(\bbR^d;\bbR^d)$ via the $L^2$ inner product, by H\"older's inequality we easily have $\Vnorm{\bu}_{[W^{t',p'}(\bbR^d)]^*} \leq \Vnorm{\bu}_{L^p(\bbR^d)}$. Since $p \in [2,2^{*_s})$ we have by log-convexity of $L^p$ norms
\begin{equation*}
\Vnorm{\bu}_{[W^{t',p'}(\bbR^d)]^*} \leq C \Vnorm{\bu}_{L^p(\bbR^d)} \leq C \Big( \Vnorm{\bu}_{L^2(\bbR^d)} + \Vnorm{\bu}_{L^{2^{*_s}}(\bbR^d)} \Big) \leq C \Vnorm{\bu}_{W^{s,2}(\bbR^d)}\,,
\end{equation*}
where the last inequality follows from Sobolev embedding. Combining the two estimates for $\bF$ and $\bu$ gives
\begin{equation}\label{eq:HigherDiffBound2}
\Vnorm{\bF + \varpi\bu}_{[W^{t',p'}(\bbR^d)]^*} \leq C \Big( \Vnorm{\bF}_{L^{2_{*_s}+\delta_0}(\bbR^d)} + \Vnorm{\bF}_{L^{2_{*_s}}(\bbR^d)} + \Vnorm{\bu}_{W^{s,2}(\bbR^d)} \Big)\,.
\end{equation}
Thus, the bound \eqref{eq:HigherDiffBound} follows from \eqref{eq:HigherDiffBound1} and \eqref{eq:HigherDiffBound2}. Putting together these inequalities we obtain that  
\begin{equation}
\Vnorm{\Upsilon^s(\bu)}_{L^p(\bbR^d)} \leq \Vnorm{\bu}_{W^{t,p}(\bbR^d)} \leq C \Big( \Vnorm{\bF}_{L^{2_{*_s}+\delta_0}(\bbR^d)} + \Vnorm{\bF}_{L^{2_{*_s}}(\bbR^d)} + \Vnorm{\bu}_{W^{s,2}(\bbR^d)} \Big)\,,
\end{equation}
as desired.
\end{proof}

\begin{remark}
One notices that in the course of the proof of Theorem \ref{Meyers-type-first step} we actually proved that weak solutions $\bu$ of the coupled nonlocal system satisfy {\em both higher integrability and higher differentiability} on a Sobolev scale. This self-improvement of solutions is of independent interest that has been studied in \cite{Auscher,Kuusi-M-S} for nonlocal elliptic scalar equations. We restrict our attention to regularity of solutions in the scale of the Bessel potential spaces.
\end{remark}

\subsection{Higher integrability when the horizon is finite}
We now address the remaining case of the proof of \autoref{Meyers-type-first step} where the horizon $0< \mathfrak{h} < \infty$, and $c_{\mathfrak{h}} \in (0, \infty)$.  We begin with the following simple observation relating the bilinear forms $\mathcal{E}_{\mathfrak{h}}$ and $\mathcal{E}$. 
\begin{proposition}\label{prop:finite-to-infinite-horizon}
For $p \in [1,\infty)$ there exists a bounded linear operator $\cP_{\mathfrak{h}}$  from $L^p(\mathbb{R}^{d};\mathbb{R}^{d})$ to itself 
such that for any $\bu, {\bv} \in W^{s, 2}(\bbR^d,\bbR^d)$, we have 
\[
\mathcal{E}_{\mathfrak{h}}(\bu, {\bv})  = c_{\mathfrak{h}} \,\mathcal{E}(\bu, {\bv}) + \int_{\bbR^{d}} \langle \cP_{\mathfrak{h}}\bu(\bdx), {\bv}(\bdx) \rangle \, \mathrm{d}\bdx.
\]
\end{proposition}
\begin{proof}
We begin by writing that 
\begin{align*}
\mathcal{E}_{\mathfrak{h}}(\bu, {\bv}) &= c_{\mathfrak{h}} \,\mathcal{E}(\bu, {\bv})\\
 &+ {1\over 2} \int_{\bbR^{d}}\int_{\bbR^{d}}k(\bx, \by)({\bf u}(\bdx)-{\bf u}(\bdy))\cdot {(\bdx-\bdy) \over |\bdx-\bdy|} ({\bf v}(\bdx)-{\bf v}(\bdy))\cdot {(\bdx-\bdy) \over |\bdx-\bdy|} \, \mathrm{d}\bx \, \mathrm{d}\by, 
\end{align*}
where $\textstyle k(\bx, \by)= c_{\mathfrak{h}} \big( 1- \chi_{B_{\mathfrak{h}} ({\bfs 0})}(|\bx-\by|) \Big) {A(\bx, \by) \over |\bx - \by|^{d + 2s}}$ is symmetric.   Notice also that $k$ is bounded and has no singularity along the diagonal $\bx=\by$.  Moreover,  it decays fast enough at infinity that for a fixed $\bx,$ the function $k(\bx, \by)$  is integrable in $\by$. We may then apply Fubini's theorem, iterate the integrals and get 
\begin{align*}
&{1\over 2} \int_{\bbR^{d}}\int_{\bbR^{d}}k(\bx, \by)({\bf u}(\bdx)-{\bf u}(\bdy))\cdot {(\bdx-\bdy) \over |\bdx-\bdy|} ({\bf v}(\bdx)-{\bf v}(\bdy))\cdot {(\bdx-\bdy) \over |\bdx-\bdy|} \, \mathrm{d}\bx \, \mathrm{d}\by \\
&= \int_{\bbR^{d}}\left(-\int_{\bbR^{d}}k(\bx, \by) {(\bdx-\bdy)\otimes (\bx-\by) \over |\bdx-\bdy|^{2}} ({\bf u}(\bdx)-{\bf u}(\bdy)) \, \mathrm{d}\by\right)  {\bf v}(\bdx)\, \mathrm{d}\bx. 
\end{align*}
We now define the operator $\cP_{\mathfrak{h}}: L^p(\bbR^d;\bbR^d) \to L^p(\bbR^d;\bbR^d)$  as
\[
\cP_{\mathfrak{h}} \bu (\bdx) = -\int_{\bbR^{d}}k(\bx, \by) {(\bdx-\bdy)\otimes (\bx-\by) \over |\bdx-\bdy|^{2}} ({\bf u}(\bdx)-{\bf u}(\bdy)) \, \mathrm{d}\by, \quad \bdx\in \bbR^d.
\]
It is clear that the operator is linear. To show that it is bounded, we notice that 
\[
\cP_{\mathfrak{h}} \bu (\bdx)  = \int_{\bbR^{d}}k(\bx, \by) {(\bdx-\bdy)\otimes (\bx-\by) \over |\bdx-\bdy|^{2}} {\bf u}(\bdy) \, \mathrm{d}\by - \bbK(\bx) {\bf u}(\bx)
\]
where easy estimates show that {$\textstyle \bbK(\bx)  = \int_{\bbR^{d}}k(\bx, \by) {(\bdx-\bdy)\otimes (\bx-\by) \over |\bdx-\bdy|^{2}} \, \mathrm{d} \by$} is a uniformly bounded matrix-valued function. The first term, on the other hand, is a convolution type operator  and its magnitude is bounded from above by the function 
\[
\alpha_2 \gamma \ast |\bu|(\bdx),\quad \text{where}\quad  \gamma({\bfs \xi}) = c_{\mathfrak{h}} \Big(1- \chi_{B_{\mathfrak{h}} ({\bfs 0})}(|{\bfs \xi}|) \Big) {1\over |{\bfs \xi}|^{d + 2s}} \in L^{1}(\mathbb{R}^{d}). 
\]
The boundedness of the operator $\cP_{\mathfrak{h}}$ on $L^{p}(\bbR^d;\bbR^d)$ now follows from Young's inequality. 
\end{proof}

Note that the same identity is true for the corresponding operators, i.e. $\bbL_{\mathfrak{h}} = c_{\mathfrak{h}} \bbL + \cP_{\mathfrak{h}}$.

\begin{remark}
It is now an easy corollary of the above proposition and Theorem \ref{theorem-korn} to state that for $0 < \mathfrak{h} \leq \infty$, 
${\bf u}\in L^{2}(\bbR^d;\bbR^{d})$,  $\mathcal{E}_{\mathfrak{h}}(\bu, {\bu}) < \infty $ if and only if $\bu \in W^{s, 2}(\bbR^d;\bbR^{d})$. 
\end{remark}

\begin{proof}[Proof of \autoref{Meyers-type-first step} when $0 < \mathfrak{h} < \infty$]
We begin with a solution $\bu$ to the system of nonlocal equations in the sense of \eqref{defnweaksoln}  corresponding to $\bF \in L^{2_{\ast_{s}} + \delta_{0}}(\bbR^d, \bbR^{d}) \cap L^{2_{\ast_{s}}}(\bbR^d, \bbR^{d})$. Using Proposition \ref{prop:finite-to-infinite-horizon}, we can conclude that $\bu$ satisfies 
\[
\mathcal{E}(\bu, {\bv}) = {1 \over c_{\mathfrak{h}}} \left(\mathcal{E}_{\mathfrak{h}}(\bu, {\bv}) - \int_{\bbR^{d}} \langle \cP_{\mathfrak{h}}\bu(\bdx), {\bv}(\bdx) \rangle d\bdx\right) 
= \langle \bF_{\mathfrak{h}}, \bv\rangle
\]
for any $\bv \in W^{s, 2}(\bbR^d,\bbR^d)$, where we have defined $\bF_{\mathfrak{h}} = {1\over c_{\mathfrak{h}} } \left(\bF + \mathcal{P}_{\mathfrak{h} }\bu \right)$.  That is, $\bu$ solves a nonlocal system corresponding to infinite horizon in the sense of \eqref{defnweaksoln} with modified right-hand side data $ \bF_{\mathfrak{h}}$.
We can therefore apply all of the arguments in the previous subsection so long as $\cP_{\mathfrak{h}}(\bu)$ belongs to the space $[W^{t',p'}(\bbR^d;\bbR^d)]^*$.  Since $p \in (2,2^{\ast_{s}})$ and ${\bu}\in L^{2^{\ast_{s}}}$ from the fractional Sobolev embedding theorem, it follows from Proposition \ref{prop:finite-to-infinite-horizon}  that $\cP_{\mathfrak{h}}(\bu) \in L^{p}(\mathbb{R}^{d};\bbR^d)$. Moreover, by our choice of $t$ and $p$, we have the estimate
\begin{align*}
\Vnorm{\cP_{\mathfrak{h}}(\bu)}_{[W^{t',p'}(\bbR^d)]^*} &\leq C \Vnorm{\cP_{\mathfrak{h}}(\bu)}_{L^p(\bbR^d)} \\
&\leq C \Big( \Vnorm{\cP_{\mathfrak{h}}(\bu)}_{L^2(\bbR^d)} + \Vnorm{\cP_{\mathfrak{h}}(\bu)}_{L^{2^{*_s}}(\bbR^d)} \Big) \\
&\leq C \left(\|\bu\|_{L^{2}(\bbR^d)} + \|\bu\|_{L^{2^{*_s}}(\bbR^d)}\right)\\
& \leq C\|\bu\|_{W^{s,2}(\bbR^d)}\,.
\end{align*}
We can now apply the higher integrability result in case of infinite horizon to conclude that $\Upsilon^s(\bu) \in L^{p}(\mathbb{R}^{d})$  with the appropriate estimates. That concludes the proof. 
\end{proof}

\section{Characterization of Potential Spaces}\label{Characterization-Section}
The main objective of this section is to prove \autoref{Character-Potential}. 
Our proof of  \autoref{Character-Potential} follows the steps presented in the proof of \cite[Theorem 1]{Stein-Bessel}.  We first develop necessary technical tools that allows us to relate the Marcinkiewicz-type integral $D^{s}({\bff})$ with the potential function of $\bff$.  
\subsection{Poisson-type kernel and integral}
We recall the standard Poisson kernel $p_t(\by)$ and introduce the modified Poisson-type kernel $\bbP_t(\by)$ given by their Fourier transforms, respectively, 
\[
\widehat{p}_t(\bfxi) = \e^{-2 \pi |\bfxi| t}\,, \qquad \widehat{\bbP}_t(\bfxi) = \e^{-2 \pi |\bfxi| t} \left( \bbI_{d+1} + (2 \pi |\bfxi| t)
	\begin{bmatrix}
		- \frac{\bfxi \otimes \bfxi}{|\bfxi|^2} & -\imath \frac{\bfxi}{|\bfxi|} \\
		-\imath \frac{\bfxi}{|\bfxi|} & 1 \\
	\end{bmatrix}
	 \right)\,.
\]
Notice that $\bbP_t(\by)$ is a $(d+1)\times (d+1)$ matrix of functions which is  explicitly given by the formula, see \cite{Scott-Mengesha}
\begin{equation}\label{explicit-for-MatrixP}
\bbP_t(\bx) = \frac{2(d+1)}{\omega_d} \frac{t}{(|\bx|^2 + t^2)^{\frac{d+3}{2}}}
\begin{bmatrix}
\bx \otimes \bx & t\bdx\\
t\bdx & t^{2}
\end{bmatrix}
\end{equation}
where $\bdx$ is considered both a column and row $d$-vector.  Several properties of the matrix kernel $\bbP_t$ are given in \cite{Scott-Mengesha}. We list now the properties that we need. First, the matrix kernel {$\mathbb{P}_{t}$} is in fact an approximation to the identity. 
For any $t>0$, if $\mathbb{I}_{d+1}$ denotes the $(d+1)\times (d+1)$  identity matrix, then  
\begin{equation}\label{Approx-identity}
\int_{\bbR^d}{\mathbb{P}_{t}(\bx)}{\, \mathrm{d}\bx} =\int_{\bbR^d}{\mathbb{P}(\bx)}{\, \mathrm{d}\bx}  = 
  \bbI_{d+1}\,.
  \end{equation}
Moreover, for each $j$, $k$, and $\ell \in \{ 1, \ldots, d, d+1 \}$ and for every $t > 0$ we have that $\p_t \mathfrak{p}^{jk}_t(\bx) \in L^1(\bbR^d)$ and $\p_{x_\ell} \mathfrak{p}^{jk}_t(\bx) \in L^1(\bbR^d)$.  We also have the following pointwise estimates: there exists a constant $c = c(d) > 0$ such that for any $j, k = 1, 2, \dots d+1$,   
\[
|\p_t \mathfrak{p}^{jk}_t(\bx)| \leq c \,|\bx|^{-d-1},\quad |\p_t \mathfrak{p}^{jk}_t(\bx)| \leq c\, t^{-d-1},\quad \forall \, \bx \in \bbR^{d},\quad t > 0. 
\]
In addition, one can easily establish  
$
|\p_{tt} \bbP_t(\bx)| \leq \frac{C}{t^{d+2}}\,$ and $ |\p_{tt} \bbP_t(\bx)| \leq \frac{C}{|\bx|^{d+2}}\,.
$ 

Throughout, functions $\bff : \bbR^d \to \bbR^{d+1}$ are of the form $\bff = (f_1, f_2, \ldots, f_d, 0)$. We treat $\bff$ both as being a vector field in $\bbR^d$ and $\bbR^{d+1}$, as well as column and row vectors; it will be clear from context.

We recall the Poisson integral of $\bff$ given by $\bu^{\Delta}(\bx,t) := p_t \, \ast \, \bff(\bx)$, where the convolution is component-wise.  A Poisson-type integral of $\bff$ can now be naturally defined using the Poisson-type kernel $\bbP_{t} = (\mathfrak{p}_t^{ij})$  as 
\[
\bU(\bx,t) := \bbP_t \, \ast \bff(\bx)\,.
\]
The convolution in the above is taken in the sense of matrix multiplication. That is, the $i^{th}$ entry component of $\bU$ is given by $U_i = \sum_{j=1}^{d + 1} \mathfrak{p}_t^{ij}\ast f_{j}$, where $\bbP_t = (\mathfrak{p}_t^{ij}).$  
Notice that taking the Fourier transform in $\bdx$ transforms the convolution into the matrix multiplication 
\begin{equation}\label{Fourier-of-U}
\widehat{{\bf U}}({\bfs \xi},t) = \widehat{\bbP}_t(\bfxi) \widehat{{\bf f}}(\bfxi)\,.
\end{equation}

The key connection of $\bbP_t$ with $D^{s}({\bff })$ is obtained through the following important relation that we will be using below. 
For any $\bz, \bx\in \bbR^{d}$, we have  
\[
\bbP_t(\bx)\left(\begin{bmatrix} \bz\\0\end{bmatrix} \right) =  \overline{\bP}(\bx, t)  \left(\bz \cdot \frac{\bx}{|\bx|}\right) \,,\qquad \bz \in \bbR^d\,,
\]
where the vector function $\overline{\bP}(\bx, t)$ is given by 
\begin{equation}\label{defn-vector-P}
\overline{\bP}(\bx, t):= \frac{2(d+1)}{\omega_d} \frac{t|\bdx|}{(|\bx|^2 + t^2)^{\frac{d+3}{2}}} \left(\begin{bmatrix} \bx\\t\end{bmatrix} \right)\,.
\end{equation}
In particular,  $\overline{\bP}(\bx, t)$ and all its derivatives in $t$ satisfy the same estimates as $\bbP_t(\bx)$ and corresponding derivatives.
Using this relation and \eqref{Approx-identity}, we see that 
\begin{align*}
\bU(\bx,t)  &= \bff(\bdx)  + \int_{\mathbb{R}^{d}} \bbP_t (\bdy)(\bff(\bdx + \bdy) - \bff(\bdx)) \, \mathrm{d}\bdy\\
& =  \bff(\bdx)  + \int_{\mathbb{R}^{d}} \overline{\bP}(\bdy, t)(\bff(\bdx + \bdy) - \bff(\bdx))\cdot {\bdy\over |\bdy|} \, \mathrm{d}\bdy\,.
\end{align*}
\subsection{Littlewood-Paley-type  $g$-function}
We can define the analogue of  the classical Littlewood-Paley $g$-function corresponding to the new Poisson-type integral  $\bU(\bx,t)$. The following definition is natural:
\begin{equation}
\mathring{\mathfrak{g}}_1(\bff)(\bx) := \left( \intdmt{0}{\infty}{t \left| \p_t \bU(\bx,t) \right|^2}{t} \right)^{1\over 2}\,.
\end{equation}
In the definition above $\nabla = (\nabla_{\bdx}, \partial_{t})$, and $ \left| \grad \bU(\bx,t) \right|^2 = |\partial_{t}\bU|^{2} + \sum_{k=1}^{d+1}|\nabla_{\bdx} U_{k}|^2$. 
We can use these functions to characterize the $L^{p}$ norm of a vector field. The following is a result similar to  \cite[Theorem 1, Chapter IV]{Stein}. 
\begin{theorem}\label{Character-Lp-norm}
Suppose that $1<p<\infty$. Then there are constants $C_1, C_2$ such that for any ${\bf f}\in L^p(\bbR^d;\bbR^{d})$ 
\[
C_1 \Vnorm{\mathring{\mathfrak{g}}_1(\bff)}_{L^p(\bbR^d)}\leq \Vnorm{\bff}_{L^p(\bbR^d)} \leq C_2 \Vnorm{\mathring{\mathfrak{g}}_1(\bff)}_{L^p(\bbR^d)}
\]
\end{theorem}
To prove the theorem we follow the steps and the approach given in \cite{Stein} for the proof of \cite[Theorem 1, Chapter IV]{Stein}.   We first prove the theorem for $p=2$.
\begin{lemma}\label{norm-of-f-p=2}
Let $\bff \in L^2(\bbR^d)$. Then $ \mathring{\mathfrak{g}}_1(\bff) \in L^2(\bbR^d)$ with
\begin{equation}\label{eq-GFxnL2-Equation2}
\Vnorm{\mathring{\mathfrak{g}}_1(\bff)}_{L^2(\bbR^d)}^2 = \frac{1}{4} \intdm{\bbR^d}{\left| \left( \bbI_d + \frac{\bfxi \otimes \bfxi}{|\bfxi|^2} \right) \widehat{\bff}(\bfxi) \right|^2}{\bfxi} \,.
\end{equation}
\end{lemma}
The proof is tedious but elementary. It is given in the appendix. 
In the next proposition we prove one of the inequalities in Theorem \ref{Character-Lp-norm}. 
\begin{proposition}\label{thm-GFxnLessThan}
Let $1<p<\infty$. If $\bff \in L^p(\bbR^d)$, then $\mathring{\mathfrak{g}}_1(\bff) \in L^p(\bbR^d)$. Moreover, there exists a positive constant depending only on $p$ and $d$ such that for all $\bff \in L^p(\bbR^d)$, 
\begin{equation}
\Vnorm{\mathring{\mathfrak{g}}_1(\bff)}_{L^p(\bbR^d)} \leq C \Vnorm{\bff}_{L^p(\bbR^d)}\,.
\end{equation}
\end{proposition}
\begin{proof}
We use the theory of singular integrals for Hilbert space-valued functions outlined in \cite[Chapter II, Section 5]{Stein}. We use the notation of that section as well.
Define the Hilbert space $\mathscr{H}$ to be the $L^2$ space on $(0,\infty)$ with the functions taking values in $\bbR^d$ with measure $t \, \mathrm{d}t$, i.e.
\[
\mathscr{H} := \left\lbrace \bh : (0,\infty) \to \bbR^{d} \, \Big| \Vnorm{\bh}_{\mathscr{H}}^2 := \int_{0}^{\infty}{t \, |\bh(t)|^2}{\, \mathrm{d}t} < \infty \right\rbrace\,.
\]
The absolute value $|{\bf h}|$ is the norm in $\bbR^{d}$. 
Denote the Banach space of bounded linear operators from $\bbR^{d}$ to $\mathscr{H}$ by $B(\bbR^d,\mathscr{H})$. 
Let $\veps > 0$ be fixed for now. For each $\bx$ consider the matrix-valued function 
$$
\cK_{\veps}(\bx, t)  := \p_t \bbP_{t+\veps}(\bx)\,.
$$
Then for any fixed $\bx \in \bbR^d$, we identify the matrix function $\cK_{\veps}(\bx,\cdot)$ by $\cK_{\veps}(\bx)$. Now we show that $ \cK_{\veps}(\bx)\in B(\bbR^d,\mathscr{H})$. This is equivalent to showing that the integral 
$
\int_{0}^{\infty}{ t \, \left| \p_t \bbP_{t+\veps}(\bx) \right|^2}{\, \mathrm{d}t}$ is finite.   Indeed,  from the formula \eqref{explicit-for-MatrixP} we have that $\left| \p_t \bbP_t(\bx) \right| \leq \frac{C}{(|\bx|^2 + t^2)^{\frac{d+1}{2}}}$, and therefore for each $\bdx\in \mathbb{R}^{d}$ we have the estimate after change of variables that 
\begin{align*}
\Vnorm{\cK_{\veps}(\bx)}_{B(\bbR^d,\mathscr{H})}^2 &= \sup_{|\by| \leq 1} \Vnorm{\Vint{\cK_{\veps}(\bx),\by}}_{\mathscr{H}}^2 \\
&\leq \sup_{|\by| \leq 1} |\by| 
\int_{0}^{\infty}{t |\p_t \bbP_{t+\veps}(\bx)|^2}{\, \mathrm{d}t} \leq C \int_{0}^{\infty}{\frac{t}{(|\bx|^2 + (t+\veps)^2)^{d+1}}}{\mathrm{d}t} \leq C_\varepsilon
\end{align*}
and
\begin{align*}
\Vnorm{\cK_{\veps}(\bx)}_{B(\bbR^d, \mathscr{H})}^2 \leq C \int_{0}^{\infty}{\frac{t}{(|\bx|^2 + (t+\veps)^2)^{d+1}}}{\, \mathrm{d}t} \leq \frac{1}{|\bx|^{2d}} \int_{0}^{\infty}{\frac{t}{ (1 + t^2)^{d+1}}{\, \mathrm{d}t}} = \frac{C}{|\bx|^{2d}}\,.
\end{align*}
From the above two estimates  we also conlcude that 
\begin{equation}\label{eq-SingularIntegral-HilbertSpaceValued1}
\bx \mapsto \Vnorm{\cK_{\veps}(\bx)}_{B(\bbR^d,\mathscr{H})} \in L^2(\bbR^d)\,.
\end{equation}
Similarly, for $1 \leq j \leq d$, again referring to \eqref{explicit-for-MatrixP} that 
\begin{equation*}
\begin{split}
\Vnorm{\p_{x_j} \cK_{\veps}(\bx)}_{B(\bbR^d,\mathscr{H})}^2 &\leq C \int_{0}^{\infty}{\frac{t}{(|\bx|^2+(t+\veps)^2)^{d+2}}}{\, \mathrm{d}t}\\
&\leq C \int_{0}^{\infty}{\frac{t}{(|\bx|^2+t^2)^{d+2}}}{\, \mathrm{d}t} \\
&= \frac{C}{|\bx|^{2d+2}}\,.
\end{split}
\end{equation*}
Thus,
\begin{equation}\label{eq-SingularIntegral-HilbertSpaceValued2}
\Vnorm{\p_{x_j} \cK_{\veps}(\bx)}_{B(\bbR^d,\mathscr{H})} \leq \frac{C}{|\bx|^{d+1}}\,, \qquad 1 \leq j \leq d\,.
\end{equation}
Now define the operator
$$
T_{\veps}(\bff)(\bx) = \int_{\bbR^d}{\cK_{\veps}(\by) \bff(\bx-\by)}{\, \mathrm{d}\by}\,.
$$
Notice that from the definition of $\cK_{\veps}(\by)$,   $T_{\veps}(\bff)(\bx)$ in terms of the Poisson-type kernel $\bU$ as $T_{\veps}(\bff)(\bx) = \partial_{t} \bU(\bdx, t+\varepsilon)$. 
It then follows that  $T_\veps$ is a vector field since the integrand is a matrix multiplying a vector.  In fact,  $T_{\veps} (\bff)(\bx)$ take their values in $\mathscr{H}$ for each $\bx \in \bbR^d$. Moreover, we have  
\begin{align*}
\Vnorm{T_{\veps}(\bff)(\bx)}_{\mathscr{H}}^2 &= \int_{0}^{\infty}{t \, |\p_t \bU(\bx,t+\veps)|^2}{\, \mathrm{d}t} \\
&= \int_{0}^{\infty}{(t+\veps) \, |\p_t \bU(\bx,t+\veps)|^2}{\, \mathrm{d}t} - \int_{0}^{\infty}{\veps \, |\p_t \bU(\bx,t+\veps)|^2}{\, \mathrm{d}t} \\
&= \int_{\veps}^{\infty}{t \, |\p_t \bU(\bx,t)|^2}{\, \mathrm{d}t} - \veps \int_{\veps}^{\infty}{ \, |\p_t \bU(\bx,t)|^2}{\, \mathrm{d}t} \\
&\leq \int_{0}^{\infty}{t \, |\p_t \bU(\bx,t)|^2}{\, \mathrm{d}t} = \big[ \mathring{\mathfrak{g}}_1(\bff)(\bx) \big]^2\,.
\end{align*}
Therefore, by the previous theorem,  $\bdx\mapsto \Vnorm{T_{\veps}(\bff)(\bx)}_{\mathscr{H}}$ is square integrable and 
\begin{align*}
\Vnorm{T_{\veps}(\bff)}_{L^2_{\bx}(\bbR^d)} \leq \frac{1}{2} \left( \intdm{\bbR^d}{\left| \left( \bbI_d + \frac{\bfxi \otimes \bfxi}{|\bfxi|^2} \right) \widehat{\bff}(\bfxi) \right|^2}{\bfxi} \right)^{1/2} \leq  \Vnorm{\bff}_{L^2(\bbR^d)}\,, 
\end{align*}
and thus, we obtain
\begin{equation}\label{eq-SingularIntegral-HilbertSpaceValued3}
\Vnorm{\widehat{\cK}_{\veps}(\bx)}_{B(\bbR^d;\mathscr{H})} \leq 1\,.
\end{equation}
Now using \eqref{eq-SingularIntegral-HilbertSpaceValued1}, \eqref{eq-SingularIntegral-HilbertSpaceValued2} and \eqref{eq-SingularIntegral-HilbertSpaceValued3} we can use the theory of singular integrals \cite[Chapter 2, Section 5]{Stein} and conclude that
$$
\Vnorm{T_{\veps}\bff}_{L^p(\bbR^d)} \leq C \Vnorm{\bff}_{L^p(\bbR^d)}\,, \qquad 1 < p < \infty\,,
$$
with $C$ independent of $\veps$. Notice from the above calculations that for each $\bdx$,  the positive function $\Vnorm{T_{\veps}(\bff)(\bx)}_{\mathscr{H}}$ increases to $\mathring{\mathfrak{g}}_1(\bff)(\bx)$ as $\veps \to 0$ and therefore, we have that
$$
\Vnorm{\mathring{\mathfrak{g}}_1(\bff)}_{L^p(\bbR^d)} \leq C \Vnorm{\bff}_{L^p(\bbR^d)}\,, \qquad 1 < p < \infty\,.
$$
That completes the proof. 
\end{proof}

Next, we prove the reverse inequality by establishing a comparison of norms of operators with matrix symbols $\bbI_{d}$ and $\bbI_{d} + \frac{\bfxi \otimes \bfxi}{|\bfxi|^2}$. We recall that for $1 \leq j \leq d$ and $f$ belonging to the class of Schwartz functions $\cS(\bbR^d)$ the $j^{th}$ \textit{Riesz transform} is defined as
\[
R_j(f)(\bx) := \frac{2}{\omega_d} \pv \int_{\bbR^d}{\frac{y_j}{|\by|^{d+1}} f(\bx-\by) }{\, \mathrm{d}\by}\,.
\]
For any $f \in \cS(\bbR^d)$ we have
$
\widehat{R_j(f)}(\bfxi) = -\imath \frac{\xi_j}{|\bfxi|} \widehat{f}(\bfxi)
$, and $\Vnorm{R_j f}_{L^p(\bbR^d)} \leq C(p) \Vnorm{f}_{L^p(\bbR^d)}$ for $1 < p < \infty$.

\begin{lemma}\label{lma-KornsForPotentials}
Let $\bff \in L^p(\bbR^d)$ for $1 < p < \infty$. The translation-invariant operator $L(\bff)$ defined by 
$$
\left[ L\bff(\bx) \right]_k := f_k(\bx) - 3 R_k \left[ \sum_{j=1}^d R_j f_j \right](\bx)\,, \qquad 1 \leq k \leq d\,,
$$
satisfies
$$
\Vnorm{\bff}_{L^p(\bbR^d)} \leq C \Vnorm{L \bff}_{L^p(\bbR^d)}\,.
$$
\end{lemma}

\begin{proof}
Clearly by the $L^p$ boundedness of the Riesz transforms,
$$
\Vnorm{L \bff}_{L^p(\bbR^d)} \leq C \Vnorm{\bff}_{L^p(\bbR^d)}\,,
$$
so the $L^p$ norm of $L \bff$ is finite. Applying the $k^{th}$ Riesz transform to $(L \bff)_k$ and summing gives
\begin{align*}
\sum_{k=1}^d R_k (L \bff)_k &= R_k f_k - 3 R_k R_k \left( \sum_{j=1}^d R_j f_j \right)= -2 \sum_{k=1}^d R_k f_k\,.
\end{align*}
Thus, by the $L^p$ boundedness of the Riesz transforms we have
\begin{equation}\label{eq-KornsInequalityForPotentials-Proof}
\Vnorm{\sum_{j=1}^d R_k f_k}_{L^p(\bbR^d)} \leq C \Vnorm{T\bff}_{L^p(\bbR^d)}\,.
\end{equation}
Now writing as 
$$
f_k(\bx) = f_k(\bx) - 3 R_k \left( \sum_{j=1}^d R_j f_j \right)(\bx) + 3 R_k \left( \sum_{j=1}^d R_j f_j \right)(\bx)\,,
$$
and so taking the $L^p$ norm on both sides and using the $L^p$ boundedness of the Riesz transforms gives
\begin{equation*}
\begin{split}
\Vnorm{f_k}_{L^p(\bbR^d)} &\leq \Vnorm{f_k - 3 R_k \left( \sum_{j=1}^d R_j f_j \right)}_{L^p(\bbR^d)} + \Vnorm{3 R_k \left( \sum_{j=1}^d R_j f_j \right)}_{L^p(\bbR^d)} \\
&= \Vnorm{(L\bff)_k}_{L^p(\bbR^d)} + \Vnorm{3 R_k \left( \sum_{j=1}^d R_j f_j \right)}_{L^p(\bbR^d)} \\
&\leq C \Vnorm{(L\bff)}_{L^p(\bbR^d)} + C \Vnorm{ \sum_{j=1}^d R_j f_j }_{L^p(\bbR^d)}  \EquationReference{\eqref{eq-KornsInequalityForPotentials-Proof}}{\leq} C \Vnorm{L \bff}_{L^p(\bbR^d)}\,.
\end{split}
\end{equation*}
Summing over $k$ finishes the proof.
\end{proof}

\begin{remark}
The symbol associated to $L$ is
$
\bbI_d + 3 \frac{\bfxi \otimes \bfxi}{|\bfxi|^2}\,,
$
which will appear in the proof of the converse inequalities for $\mathring{\mathfrak{g}}_1$. 
\end{remark}
The next result proves the remaining inequality in Theorem \ref{Character-Lp-norm}. 
\begin{proposition}\label{character-via-g1}
Let $1<p<\infty$. Then there exists a positive constant $C$ such that for any $\bff \in L^p(\bbR^d;\bbR^{d})$,  
$$
\Vnorm{\bff}_{L^p(\bbR^d)} \leq C \Vnorm{\mathring{\mathfrak{g}}_1(\bff)}_{L^p(\bbR^d)}\,.
$$
\end{proposition}
\begin{proof}
Let $\bff_1$, $\bff_2$ be in $L^2(\bbR^d;\bbR^{d})$ with respective Poisson-type integrals $\bU_1$, $\bU_2$. Polarization of the identity \eqref{eq-GFxnL2-Equation2} leads to
{\begin{equation}
\begin{split}
&\int_{0}^{\infty}\int_{\bbR^d}{t \Vint{\p_t \bU_1(\bx,t), \p_t \bU_2(\bx,t)}}{d\bx}{dt} \\
&= \frac{1}{4} \intdm{\bbR^d}{\Vint{\left( \bbI_d + \frac{\bfxi \otimes \bfxi}{|\bfxi|^2} \right)\widehat{\bff_1}(\bfxi), \left( \bbI_d + \frac{\bfxi \otimes \bfxi}{|\bfxi|^2} \right)\widehat{\bff_2}(\bfxi)}}{\bfxi} \\
&= \frac{1}{4} \intdm{\bbR^d}{\Vint{\left( \bbI_d + 3 \frac{\bfxi \otimes \bfxi}{|\bfxi|^2} \right)\widehat{\bff_1}(\bfxi), \widehat{\bff_2}(\bfxi)}}{\bfxi} \\
&= \frac{1}{4} \intdm{\bbR^d}{\Vint{L\bff_1(\bx),\bff_2(\bx)}}{\bx}\,,
\end{split}
\end{equation}}
where the last inequality follows by Parseval's relation.
Now suppose in addition that $\bff_1 \in L^p(\bbR^d;\bbR^{d})$ and $\bff_2 \in L^{p'}(\bbR^d;\mathbb{R}^{d})$ with $\Vnorm{\bff_2}_{L^{p'}(\bbR^d)} \leq 1$. Then using the Cauchy-Schwarz inequality, H\"older's inequality and Proposition \ref{thm-GFxnLessThan} we have
\begin{align*}
\left| \int_{\bbR^d}{\Vint{L\bff_1(\bx),\bff_2(\bx)}}{\, \mathrm{d}\bx} \right|& \leq 4 \int_{\bbR^d}{\mathring{\mathfrak{g}}_1(\bff_1)(\bx) \mathring{\mathfrak{g}}_1(\bff_2)(\bx)}{\, \mathrm{d}\bx} \\
&\leq 4 \Vnorm{\mathring{\mathfrak{g}}_1(\bff_1)}_{L^p(\bbR^d)} \Vnorm{\mathring{\mathfrak{g}}_1(\bff_2)}_{L^{p'}(\bbR^d)} \leq C \Vnorm{\mathring{\mathfrak{g}}_1(\bff_1)}_{L^p(\bbR^d)}\,.
\end{align*}
Taking the supremum on both sides over all $\bff_2 \in L^2(\bbR^{d};\bbR^{d}) \cap L^{p'}(\bbR^d;\bbR^{d})$ with $\Vnorm{\bff_2}_{L^{p'}(\bbR^d)} \leq 1$ gives
\begin{equation}
\Vnorm{L \bff}_{L^p(\bbR^d)} \leq C \Vnorm{\mathring{\mathfrak{g}}_1(\bff)}_{L^p(\bbR^d)}
\end{equation}
for every $\bff \in (L^2 \cap L^p)(\bbR^d)$. Using Lemma \ref{lma-KornsForPotentials},
\begin{equation}\label{eq-GFxnGreaterThan-Proof}
\Vnorm{\bff}_{L^p(\bbR^d)} \leq C \Vnorm{\mathring{\mathfrak{g}}_1(\bff)}_{L^p(\bbR^d)}
\end{equation}
for every $\bff \in (L^2 \cap L^p)(\bbR^d)$.
The passage to the general case $\bff \in L^p$ follows by density.  Let $\bff_m$ be a sequence of functions in $(L^2 \cap L^p)(\bbR^d)$ which converge in $L^p$ to an arbitrary function $\bff \in L^p$. Then
\begin{align*}
\left| \mathring{\mathfrak{g}}_1(\bff_m)(\bx) - \mathring{\mathfrak{g}}_1(\bff)(\bx) \right|^2 &= \left| \left( \intdmt{0}{\infty}{t \, |\p_t \bU_m (\bx,t)|^2}{t} \right)^{\frac{1}{2}} - \left( \intdmt{0}{\infty}{t \, |\p_t \bU (\bx,t)|^2}{t}  \right)^{\frac{1}{2}} \right|^2 \\
	&= \intdmt{0}{\infty}{t \, \left( |\p_t \bU_m(\bx,t)|^2 + |\p_t \bU(\bx,t)|^2 \right)}{t} \\
	&\qquad-2 \left( \intdmt{0}{\infty}{t \, |\p_t \bU_m (\bx,t)|^2}{t} \right)^{\frac{1}{2}} \left( \intdmt{0}{\infty}{t \, |\p_t \bU (\bx,t)|^2}{t} \right)^{\frac{1}{2}}.  
	\end{align*}
By Cauchy-Schwartz inequality the last expression cannot exceed 
\[
\int_{0}^{\infty}{t \, \left( |\p_t \bU_m(\bx,t)|^2 + |\p_t \bU(\bx,t)|^2 - 2 \Vint{\p_t \bU_m(\bx,t), \p_t \bU(\bx,t)}\right) \,}{\mathrm{d}t}\,,
\]
and so 
\[
\left| \mathring{\mathfrak{g}}_1(\bff_m)(\bx) - \mathring{\mathfrak{g}}_1(\bff)(\bx) \right|^2  \leq \int_{0}^{\infty}{t \, |\p_t (\bU_m - \bU)(\bx,t)|^2 \,}{\mathrm{d}t} = \left| \mathring{\mathfrak{g}}_1(\bff_m - \bff)(\bx) \right|^2. 
\]
Therefore by Theorem \ref{thm-GFxnLessThan},  $\mathring{\mathfrak{g}}_1(\bff_m)$ converges to $\mathring{\mathfrak{g}}_1(\bff)$ in $L^p(\bbR^d)$, so we obtain \eqref{eq-GFxnGreaterThan-Proof} for a general $\bff \in L^p(\bbR^d)$ and the proof is complete.
\end{proof}

Now we come to the final preliminary inequality that must be established before proving our main result. 
Recall that the Riesz potential $\cI^s$  and the Bessel potential $\cJ^s$  acting on a function $\bff$ 
  are given by 
\begin{equation}
\cI_s(\bff)(\bx) := c_{d,s} \intdm{}{\frac{\bff(\by)}{|\bx-\by|^{d-s}}}{\by}\,,
\end{equation}
\begin{equation}
\cJ_s(\bff)(\bx) := ( \widehat{\cG_s} \, \widehat{\bff} )^{\vee} = \cG_s \, \ast \, \bff\,, \qquad \cG_s(\bx) := \big( (1+4 \pi^2 |\bfxi|^2)^{-s/2} \big)^{\vee}(\bx)\,,
\end{equation}
where $c_{d,s}$ is an appropriate normalizing constant.

\begin{theorem}\label{theorem-PointwiseEstimateByDs}
Let $\bff \in L^p(\bbR^d)$, $1 < p < \infty$ and let $0 < s < 1$. Denote $\bff_s := \cI_s(\bff)$. Let $\bU$ and $\bU_s$ be the Poisson-type integrals of $\bff$, $\bff_s$ respectively. Then for every $\bx \in \bbR^d$ we have
\[
\mathring{\mathfrak{g}}_1(\bff)(\bx) \leq C \,  D^{s}(\bff_s)(\bx)\,.
\]
\end{theorem}
\begin{proof}
We first establish the equality
\begin{equation}\label{relation-Ut-Utt}
\p_t \bU(\bx,t) = \frac{-1}{\Gamma(1-s)} \int_{0}^{\infty}{\p_{tt} \bU_s(\bx,t+r)r^{-s}}{\mathrm{d}r}\,,
\end{equation}
where $\Gamma$ denotes the Gamma function.  
This identity can be established using the Fourier transform and will be done in the appendix. 
We will use this to estimate $\mathring{\mathfrak{g}}_1(\bff)$ pointwise.  To that end, we write 
\begin{equation*}
\begin{split}
\mathring{\mathfrak{g}}_1(\bff)(\bx) = \intdmt{0}{\infty}{t \, |\p_t \bU(\bx,t) |^2}{t}= \intdmt{0}{\infty}{t \, \left| \frac{-1}{\Gamma(1-s)} \intdmt{t}{\infty}{\p_{tt} \bU_s(\bx,r)(r-t)^{-s}}{r} \right|^2}{t}. 
\end{split}
\end{equation*}
Dividing the intervals of integration in the inside integral, we see that 
\begin{equation*}
\begin{split}
	\mathring{\mathfrak{g}}_1(\bff)(\bx)&\leq C \intdmt{0}{\infty}{t \, \left| \intdmt{t}{2t}{\p_{tt} \bU_s(\bx,r)(r-t)^{-s}}{r} \right|^2}{t} \\
	&+ C \intdmt{0}{\infty}{t \, \left| \intdmt{2t}{\infty}{\p_{tt} \bU_s(\bx,r)(r-t)^{\frac{1}{2}} (r-t)^{-\frac{1}{2}-s}}{r} \right|^2}{t} \\
	&\leq C \intdmt{0}{\infty}{t \, \left( \intdmt{t}{2t}{|\p_{tt} \bU_s(\bx,r)|^2 (r-t)^{-s}}{r} \right) \left( \intdmt{t}{2t}{\frac{1}{(r-t)^{s}}}{r} \right)}{t} \\
	&\qquad + C \intdmt{0}{\infty}{t \, \left( \intdmt{2t}{\infty}{|\p_{tt} \bU_s(\bx,r)|^2 (r-t)}{r} \right) \left( \intdmt{2t}{\infty}{\frac{1}{(r-t)^{1+2s}}}{r} \right)}{t},  \end{split}
	\end{equation*}
where we have used the Cauchy-Schwartz inequality in the last inequality.  Simplification and interchanging of the integrals via Fubini implies that  
\begin{align*}\label{g-bounded-by-Utt}
	\mathring{\mathfrak{g}}_1(\bff)(\bx)&\leq C \intdmt{0}{\infty}{t^{2-s} \, \intdmt{t}{2t}{|\p_{tt} \bU_s(\bx,r)|^2 (r-t)^{-s}}{r}}{t} \\
	&+ C \intdmt{0}{\infty}{t^{1-2s} \, \intdmt{2t}{\infty}{|\p_{tt} \bU_s(\bx,r)|^2 (r-t)}{r}}{t} \\
	&=C \intdmt{0}{\infty}{|\p_{tt} \bU_s(\bx,r)|^2 \left( \intdmt{r/2}{r}{t^{2-s}(r-t)^{-s}}{t} + \intdmt{0}{r/2}{t^{1-2s}(r-t)}{t}\right)}{r} \\
	&\leq C \intdmt{0}{\infty}{r^{3-2s} |\p_{tt} \bU_s(\bx,r)|^2}{r} = C \intdmt{0}{\infty}{t^{3-2s} |\p_{tt} \bU_s(\bx,t)|^2}{t}\,.
\end{align*}
Notice also that using the fact that $\bbP_t$ integrates to $\bbI_{d+1}$ we see that 
\begin{align*}
\p_{tt} \bU_s(\bx,t) &= \intdm{\bbR^d}{\p_{tt}\bbP_t(\by)(\bff_s(\bx-\by)-\bff_s(\bx))}{\by} \\
&= \intdm{\bbR^d}{\p_{tt}\overline{\bP}(\by, t)(\bff_s(\bx-\by)-\bff_s(\bx)) \cdot \frac{\by}{|\by|}}{\by}\,,
\end{align*}
where we use the relation \eqref{defn-vector-P}.
By computation of $\p_{tt}\bbP_t $ it is not difficult to show that  
\[
|\p_{tt} \bU_s(\bx,t)| \leq \int_{\bbR^{d}}{|\p_{tt} \overline{\bP}(\by,t)| \, \left| (\bff_s(\bx+\by)-\bff_s(\bx)) \cdot \frac{\by}{|\by|}\right| }{\, \mathrm{d}\by}\,. 
\]
We now divide the integration region in the right hand side and estimate using H\"older's inequality to obtain  
\begin{align*}
|\p_{tt} \bU_s(\bx,t)|^2 &\leq \left( \int_{|\by| \leq t}{|\p_{tt} \overline{\bP}(\by,t)| \, \left| \bff_s(\bx+\by)-\bff_s(\bx) \cdot \frac{\by}{|\by|}\right| }{\, \mathrm{d}\by} + \int_{t \leq |\by|}{\cdots}{\, \mathrm{d}\by} \right)^2 \\
	&\leq \left( \int_{|\by| \leq t}{|\p_{tt} \overline{\bP}(\by, t)|}{\, \mathrm{d}\by} \right) \left( \int_{|\by| \leq t}{|\mathfrak{d}(\bff_s)(\bx,\by)|^2 \, |\p_{tt} \overline{\bP}_t(\by)|}{\, \mathrm{d}\by} \right) \\
	& \qquad + \left( \int_{|\by| > t}{|\p_{tt} \overline{\bP}(\by, t)|}{\, \mathrm{d}\by} \right) \left( \int_{|\by| > t}{|\mathfrak{d}(\bff_s)(\bx,\by)|^2 \, |\p_{tt} \overline{\bP}(\by, t)|}{\, \mathrm{d}\by} \right)\,. 
	\end{align*}
	where we introduced the notation $\mathfrak{d}(\bff_s)(\bx,\by) := \bff_s(\bx+\by)-\bff_s(\bx) \cdot \frac{|\by|}{|\by|}$.
	We now use the estimates for $|\p_{tt} \overline{\bP}(\by,t)|$ to get 
\begin{align*}
|\p_{tt} \bU_s(\bx,t)|^2 
	&\leq \left( \int_{|\by| \leq t}{ \frac{C}{t^{d+2}}}{\, \mathrm{d}\by} \right) \left( \int_{|\by| \leq t}{|\mathfrak{d}(\bff_s)(\bx,\by)|^2 \, |\p_{tt} \overline{\bP}(\by,t)|}{\, \mathrm{d}\by} \right) \\
	& \qquad + \left( \int_{|\by| > t}{ \frac{C}{|\by|^{d+2}}}{\, \mathrm{d}\by} \right) \left( \int_{|\by| > t}{|\mathfrak{d}(\bff_s)(\bx,\by)|^2 \, |\p_{tt} \overline{\bP}(\by, t)|}{\, \mathrm{d}\by} \right) \\
	&\leq \frac{C}{t^2} \left( \int_{|\by| \leq t}{|\mathfrak{d}(\bff_s)(\bx,\by)|^2 \, |\p_{tt} \overline{\bP}(\by, t)|}{\, \mathrm{d}\by} \right) \\
	&+ \frac{C}{t^2} \left( \int_{|\by| > t}{|\mathfrak{d}(\bff_s)(\bx,\by)|^2 \, |\p_{tt} \overline{\bP}(\by, t)|}{\, \mathrm{d}\by} \right) \,. 
\end{align*}
As a consequence, combining the above with \eqref{g-bounded-by-Utt}, and interchanging the integrals we have that 
{ \begin{align*}
\big[ \mathring{\mathfrak{g}}_1(\bff)(\bx) \big]^2  &
\leq C\int_{0}^{\infty} t^{3-2s} \left( \frac{1}{t^2}\right) \Bigg( \int_{|\by| \leq t}{|\mathfrak{d}(\bff_s)(\bx,\by)|^2 \, |\p_{tt} \overline{\bP}(\by, t)|}{\, \mathrm{d}\by} \\ 
	&\qquad \qquad \qquad \qquad \qquad \qquad \qquad  + \int_{|\by| > t}{\ldots}{\, \mathrm{d}\by} \Bigg) {\, \mathrm{d}t} \\
	&= C \int_{\bbR^d} |\mathfrak{d}(\bff_s)(\bx,\by)|^2  \Bigg( \int_{0}^{|\by|}{t^{1-2s} |\p_{tt} \overline{\bP}(\by, t)|}{\, \mathrm{d}t} \\
	&\qquad \qquad \qquad \qquad \qquad \qquad \qquad + \int_{|\by|}^{\infty}{t^{1-2s} |\p_{tt} \overline{\bP}(\by,t)|}{\, \mathrm{d}t} \Bigg) {\, \mathrm{d}\by} \\
	&\leq C \int_{\bbR^d}{ |\mathfrak{d}(\bff_s)(\bx,\by)|^2  \left( \int_{0}^{|\by|}{\frac{t^{1-2s}}{|\by|^{d+2}}}{\, \mathrm{d}t} + \int_{|\by|}^{\infty}{t^{-d-1-2s}}{\, \mathrm{d}t} \right)}{\, \mathrm{d} \by} \\
	&= C \int_{\bbR^d}{\frac{|\mathfrak{d}(\bff_s)(\bx,\by)|^2}{|\by|^{d+2s}}}{\, \mathrm{d}\by} = C \big[ D^{s} (\bff_s)(\bx) \big]^2 \,.
\end{align*}}
The proof is complete.
\end{proof}
Finally we are ready to prove the second main result of the paper, Theorem \ref{Character-Potential}. As we have indicated since for any $\bdx\in \bbR^{d}$ the pointwise estimate $D^{s}({\bff})(\bx) \leq \Upsilon^{s}(\bff)(\bdx)$ holds, the right-hand side inequality in Theorem \ref{Character-Potential} follows from the characterization in \cite[Theorem 1]{Stein-Bessel}. What remains is to prove the left-hand side inequality in Theorem \ref{Character-Potential} which is stated in the following theorem. 
\begin{theorem}
Let $s\in (0, 1)$ and $1 < p < \infty$. If  $\bff \in \mathcal{D}^{p}_{s}(\mathbb{R}^{d})$, then $\bff \in \mathcal{L}^{s,p}(\bbR^d)$. Moreover, there exist positive constants $C$  such that 
\[
 \| \bff\|_{\cL^{p}_{s}(\bbR^d)} \leq C\left(\|{\bf f}\|_{L^{p}(\bbR^d)} + \|D^{s}({\bf u}) \|_{L^{p}(\bbR^d)} \right)
\]

\end{theorem}
\begin{proof}
It suffices to establish the inequality for vector fields in the Schwarz space $ \mathcal{S}(\bbR^d;\bbR^d)$.
To that end, we will show that for any $\bff  \in \mathcal{S}(\bbR^d;\bbR^d)$, the tempered distribution  $\left( \left( 1 + 4 \pi^2 |\bfxi|^2 \right)^{s/2} \widehat{\bff} \right)^{\vee}$ belongs to the space $ L^p(\bbR^d;\mathbb{R}^{d})\, $ with the appropriate estimate. 
In what follows we use a key result that relates the Riesz  potentials $\mathcal{I}_{s}$ and the Bessel potentials $\mathcal{J}_{s}$. From \cite[Lemma 2 of Chapter V]{Stein} there exists a pair of finite measures $\nu_s$ and $\lambda_s$ on $\bbR^{d}$ with Fourier transforms $\widehat{\nu}_s(\bfxi)$ and $\widehat{\lambda}_s(\bfxi)$ respectively such that 
\[
\left( 1 + 4 \pi^2 |\bfxi|^2 \right)^{s/2}  = \widehat{\nu}_s(\bfxi) + (2 \pi |\bfxi|)^s \widehat{\lambda}_s(\bfxi)\,.
\] 
Then for any $\bff\in L^{p}(\mathbb{R}^{d};\bbR^{d})$, we have that in the sense of distributions  
\[
\left( \left( 1 + 4 \pi^2 |\bfxi|^2 \right)^{s/2} \widehat{\bff} \right)^{\vee} = \bff\ast\nu_s + ((2\pi |\xi|)^{s}\widehat{\bff})^{\vee} \ast \lambda_s
\]
We now estimate the $L^{p}$ norms of the terms in the right-hand side. We notice first that $\|\bff\ast\nu_s\|_{L^{p}} \leq C\|{\bff }\|_{L^{p}}$, which follows since $\nu_s$ is a finite measure. To estimate the second term we use again the fact that $\lambda_s$ is a finite measure to get
\[
\Vnorm{((2\pi |\xi|)^{s}\widehat{\bff})^{\vee} \ast \lambda_s}_{L^{p}} \leq C \Vnorm{((2\pi |\xi|)^{s}\widehat{\bff})^{\vee}}_{L^{p}} \leq C\Vnorm{\mathring{\mathfrak{g}}_1\Big( \, \big( (2\pi |\xi|)^{s}\widehat{\bff} \big)^{\vee} \, \Big)}_{L^{p}}\,,
\]
where in the last inequality we have applied Proposition \ref{character-via-g1}. We now use Theorem \ref{theorem-PointwiseEstimateByDs} to estimate 
\begin{align*}
\Vnorm{((2\pi |\xi|)^{s}\widehat{\bff})^{\vee} \ast \lambda_s}_{L^{p}} &\leq C\Vnorm{\mathring{\mathfrak{g}}_1\Big( \, \big( (2\pi |\xi|)^{s}\widehat{\bff} \big)^{\vee} \, \Big)}_{L^{p}}\\
&\leq C \Vnorm{D^{s} \bigg( \mathcal{I}_{s} \Big( \big( (2\pi |\xi|)^{s}\widehat{\bff} \big)^{\vee} \Big) \bigg)}_{L^{p}} \\
&= C\|D^{s}(\bff)\|_{L^{p}}\,,
\end{align*}
where we used the identity $ \mathcal{I}_{s}\Big( \big( (2\pi |\xi|)^{s}\widehat{\bff} \big)^{\vee} \Big) = {\bff }$.  
\end{proof}
\section{Conclusion} In this paper we have established a qualitative property of solutions to a strongly coupled system of nonlocal equations that arise from the linearization of the bond-based peridynamic model. We have obtained a higher integrability potential space estimate for solutions corresponding to measurable and elliptic  coefficients and possibly rough data.  Our proof of the result adapts arguments from the functional analytic approach developed in \cite{Auscher} considering regularity of scalar-valued nonlocal elliptic equations.  
The regularity estimate, which should be considered as a nonlocal analogue of the celebrated  inequality of Meyers that applies to system of PDEs with elliptic and measurable coefficients, is applicable  to those nonlocal models with kernel that is locally comparable with fractional kernels of type  $|\xi|^{-(d + 2s)}$ and with bounded and unbounded support. In particular, it is applicable for peridynamic models with fractional kernels and finite horizon.  

It is anticipated that the higher integrability  result we obtained in the current work will be used to obtain estimates for solutions of nonlocal equations with coefficients that have large jump discontinuities  as well as highly oscillatory features.   For example, there is an interest in such types of estimates  for the homogenization of peridynamic models with  highly oscillatory coefficients. The higher integrability of a measure of smoothness of a sequence of solutions guarantees compactness of the solutions in $L^{2}$ for example, and assists in establishing convergence rates of solutions to the homogenized solution. We hope to make a rigorous study of this in a future work.  

\appendix
\section{Proof of Lemma \ref{norm-of-f-p=2}}
\begin{proof} 
Using Fubini's Theorem and Plancherel's Theorem,
\[
\Vnorm{\mathring{g}_{1}(\bff)}_{L^2(\bbR^d)}^2 = \intdmt{0}{\infty}{t \intdm{\bbR^d}{|\p_t \bU(\bx,t)|^2}{\bx}}{t} = \intdmt{0}{\infty}{t \intdm{\bbR^d}{|\widehat{\p_t\bU}(\bfxi,t)|^2}{\bfxi}}{t}\,.
\]
By a direct computation, for $1 \leq k \leq d$,
\begin{align*}
|\p_t \widehat{U_k}(\bfxi,t) |^2 &= \left| \p_t \left( \e^{-2 \pi |\bfxi| t} \widehat{f}_k + (2 \pi |\bfxi| t) \e^{-2 \pi |\bfxi| t} \left( \frac{\bfxi}{|\bfxi|} \cdot \widehat{\bff}(\bfxi) \right) \left( \frac{-\xi_k}{|\bfxi|} \right) \right) \right|^2 \\
	& = 4 \pi^2 |\bfxi|^2 \e^{-4 \pi |\bfxi| t} \left| A_k(\bfxi,t) \right|^2 \,,
\end{align*}
where
$$
A_k(\bfxi,t) := \widehat{f}_k(\bfxi) + \left( \frac{\bfxi}{|\bfxi|} \cdot \widehat{\bff}(\bfxi) \right) \left( \frac{\xi_k}{|\bfxi|} \right) - (2 \pi |\bfxi| t) \left( \frac{\bfxi}{|\bfxi|} \cdot \widehat{\bff}(\bfxi) \right) \left( \frac{\xi_k}{|\bfxi|} \right)\,.
$$
(Note that here we are using the modulus $| \cdot |$ for complex numbers.) Next,
\begin{align*}
\left| \p_t \widehat{U_{d+1}}(\bfxi,t) \right|^2 &= \left| \p_t \left( (2 \pi |\bfxi| t) \e^{-2 \pi |\bfxi| t} \left( -\imath \frac{\bfxi}{|\bfxi|} \cdot \widehat{\bff}(\bfxi) \right) \right) \right|^2 \\
&= 4 \pi^2 |\bfxi|^2 \e^{-4 \pi |\bfxi| t} \left| \frac{\bfxi}{|\bfxi|} \cdot \widehat{\bff}(\bfxi) \right|^2 (1 + 2 \pi |\bfxi| t)^2\,,
\end{align*}
Expanding the $A_k$ terms,
\begin{align*}
\sum_{k=1}^d \left| A_k(\bfxi,t) \right|^2 = |\widehat{\bff}(\bfxi)|^2 &+ 3 \left( \frac{\bfxi}{|\bfxi|} \cdot \widehat{\bff}(\bfxi) \right)^2 - (8 \pi |\bfxi| t)  \left( \frac{\bfxi}{|\bfxi|} \cdot \widehat{\bff}(\bfxi) \right)^2\\
& + (4 \pi^2 |\bfxi|^2 t^2)  \left( \frac{\bfxi}{|\bfxi|} \cdot \widehat{\bff}(\bfxi) \right)^2\,.
\end{align*}
Now we prove \eqref{eq-GFxnL2-Equation2}. By the above formulas,
{\begin{align*}
&\left| \p_t \widehat{\bU}(\bfxi,t) \right|^2 \\
&= 4 \pi^2 |\bfxi|^2 \e^{-4 \pi |\bfxi| t} \left( \sum_{k=1}^d \left| A_k(\bfxi,t)\right|^2 + (1 + 2 \pi |\bfxi| t)^2 \left| \frac{\bfxi}{|\bfxi|} \cdot \widehat{\bff}(\bfxi) \right|^2 \right) \\
	&= 4 \pi^2 |\bfxi|^2 \e^{-4 \pi |\bfxi| t} \left( |\widehat{\bff}(\bfxi)|^2 + \left( 3 - 8 \pi |\bfxi| t + 4 \pi^2 |\bfxi|^2 t^2 + (1 + 2 \pi |\bfxi| t)^2 \right) \left| \frac{\bfxi}{|\bfxi|} \cdot \widehat{\bff}(\bfxi) \right|^2 \right) \\
	& = 4 \pi^2 |\bfxi|^2 \e^{-4 \pi |\bfxi| t} |\widehat{\bff}(\bfxi)|^2+ 4 \pi^2 |\bfxi|^2 \e^{-4 \pi |\bfxi| t} \Big( 4 - 4 \pi |\bfxi| t +8 \pi^2 |\bfxi|^2 t^2  \Big) \left| \frac{\bfxi}{|\bfxi|} \cdot \widehat{\bff}(\bfxi) \right|^2
\end{align*}}
Thus,
\begin{align*}
\iintdmt{0}{\infty}{\bbR^d}{t | \p_t \widehat{\bU}(\bfxi,t)|^2}{\bfxi}{t} &= \intdm{\bbR^d}{\frac{1}{4}|\widehat{\bff}(\bfxi)|^2 +\left( \frac{1}{2} - \frac{1}{2} + \frac{3}{4} \right) \left|\frac{\bfxi}{|\bfxi|} \cdot \widehat{\bff}(\bfxi) \right|^2 }{\bfxi} \\
&= \frac{1}{4} \intdm{\bbR^d}{\left| \left( \bbI_d + \frac{\bfxi \otimes \bfxi}{|\bfxi|^2} \right) \widehat{\bff}(\bfxi) \right|^2}{\bfxi}\,.
\end{align*}
where we have used the formulas: by a change of variables $\eta = 4 \pi |\bfxi| t$,
\begin{align*}
\intdmt{0}{\infty}{8 \pi^2 |\bfxi|^2 t \e^{-4 \pi |\bfxi| t}}{t} &= \frac{1}{2} \intdmt{0}{\infty}{\eta \e^{-\eta}}{\eta} = \frac{1}{2} \Gamma(1) = \frac{1}{2}\,,\\
\intdmt{0}{\infty}{16 \pi^2 |\bfxi|^2 t \e^{-4 \pi |\bfxi| t}}{t} &= \intdmt{0}{\infty}{\eta \e^{-\eta}}{\eta} = 1\,,\\
- \intdmt{0}{\infty}{32 \pi^3 |\bfxi|^3 t^2 \e^{-4 \pi |\bfxi| t}}{t} &= - \frac{1}{2} \intdmt{0}{\infty}{\eta^2 \e^{\eta}}{\eta} = - \frac{\Gamma(2)}{2} = -1\,,\\
\intdmt{0}{\infty}{64 \pi^4 |\bfxi|^4 t^3 \e^{-4 \pi |\bfxi| t}}{t} &= \frac{1}{4} \intdmt{0}{\infty}{\eta^3 \e^{-\eta}}{t} = \frac{\Gamma(3)}{4} = \frac{3}{2}\,.
\end{align*}

\section{Proof of \eqref{relation-Ut-Utt} }

This can be done using the Fourier transform. Denote
$$
 \bbA(\bfxi)\, =: \begin{bmatrix}
		- \frac{\bfxi \otimes \bfxi}{|\bfxi|^2} & -i \frac{\bfxi}{|\bfxi|} \\
		-i \frac{\bfxi}{|\bfxi|} & 1 \\
	\end{bmatrix}
	.
$$
Then
{\begin{align*}
&\frac{-1}{\Gamma(1-s)} \intdmt{0}{\infty}{ \p_{tt} \widehat{\bU_s}(\bfxi,t+r)r^{-s}}{r} \\
	&= \frac{-1}{\Gamma(1-s)} \int_{0}^{\infty} 4 \pi^2 |\bfxi|^2 \e^{-2 \pi |\bfxi| (t+r)} \Big( \bbI_{d+1} - 2 \bbA(\bfxi) + 2 \pi |\bfxi|(t+r) \bbA(\bfxi) \Big) \\
	&\qquad \qquad \qquad \qquad \qquad \qquad \qquad \qquad \qquad \times (2 \pi |\bfxi|)^{-s} \widehat{\bff}(\bfxi) r^{-s} \, \mathrm{d}r \\
	&= \frac{-1}{\Gamma(1-s)} 2 \pi |\bfxi| \e^{-2 \pi |\bfxi| t} \left( \intdmt{0}{\infty}{(2 \pi |\bfxi| r)^{-s} \e^{-2 \pi |\bfxi| r} (2 \pi |\bfxi|) }{r} \right) \\
	&\qquad \qquad \qquad \qquad \qquad \qquad \qquad \qquad \qquad \times \bigg( \bbI_{d+1} - 2 \bbA(\bfxi) + 2 \pi |\bfxi| t \bbA(\bfxi) \bigg) \\
	&\qquad + \frac{-1}{\Gamma(1-s)} 2 \pi |\bfxi| \e^{-2 \pi |\bfxi| t} \left( \intdmt{0}{\infty}{(2 \pi |\bfxi| r)^{1-s} \e^{-2 \pi |\bfxi| r} (2 \pi |\bfxi|) }{r}  \right) \bbA(\bfxi) \\
	&= \frac{-1}{\Gamma(1-s)} 2 \pi |\bfxi| \e^{-2 \pi |\bfxi| t} \, \Gamma(1-s) \, \bigg( \bbI_{d+1} - 2 \bbA(\bfxi) + 2 \pi |\bfxi| t \bbA(\bfxi) \bigg) \\
	&\qquad + \frac{-1}{\Gamma(1-s)} 2 \pi |\bfxi| \e^{-2 \pi |\bfxi| t} \, \Gamma(2-s) \, \bbA(\bfxi) \\
	&= (-2 \pi |\bfxi|) \e^{-2 \pi |\bfxi| t} \bigg( \bbI_{d+1} - 2 \bbA(\bfxi) + (2 \pi |\bfxi| t) \bbA(\bfxi) + \frac{\Gamma(2-s)}{\Gamma(1-s)} \bbA(\bfxi) \bigg) \\&= \p_t \widehat{\bU}(\bfxi,t)\,.
\end{align*}}
In the last equality we used the identity $\Gamma(x+1) = x \Gamma(x)$ for every $x > 0$.
By a change of variables we have
$$
\p_t \bU(\bx,t) = \frac{-1}{\Gamma(1-s)} \intdmt{t}{\infty}{\p_{tt} \bU_s(\bx,r)(r-t)^{-s}}{r}\,.
$$

\end{proof}

\end{document}